\pdfpageattr{/Group <</S /Transparency /I true /CS /DeviceRGB>>}
\documentclass[letterpaper, twoside, 10pt, leqno]{amsart}

\usepackage{amssymb}
\usepackage{mathtools}
\usepackage{mathrsfs}
\usepackage{tikz}
    \usetikzlibrary{cd}
\usepackage{enumitem}
    \setlist[enumerate]{label={\textnormal{(\alph*)~\,}}, nosep, align=left, left={\parindent}, labelsep=0pt, widest=b, rightmargin=0pt}
\usepackage[alphabetic, nobysame]{amsrefs}
\usepackage[colorlinks=true, allcolors=blue]{hyperref}
\usepackage[hyphenbreaks]{breakurl}

\allowdisplaybreaks
\numberwithin{equation}{section}

\theoremstyle{definition}
    \newtheorem{pdef}{Definition}[section]
    \newtheorem{pex}{Example}[section]
\theoremstyle{plain}
    \newtheorem{pthrm}{Theorem}[section]
    \newtheorem{pmain}{Theorem}
    \newtheorem{plem}[pthrm]{Lemma}
    \newtheorem{pcor}[pthrm]{Corollary}
    \newtheorem{pprop}[pthrm]{Proposition}
\theoremstyle{remark}
    \newtheorem{prmrk}{Remark}[section]

\DeclareMathOperator{\sech}{sech}

\newcommand{\injto}{\hookrightarrow}
\newcommand{\surjto}{\twoheadrightarrow}

\newcommand{\isoto}[1][\sim]{\stackrel{#1}{\to}}

\newcommand{\longbijto}{\longleftrightarrow}

\newcommand{\del}{\nabla}

\newcommand{\Ric}{\mathrm{Ric}}
\DeclareMathOperator{\tr}{tr}

\DeclareMathOperator{\Ad}{Ad}

\newcommand{\II}{\mathrm{I\!\!\:I}}
\newcommand{\III}{\mathrm{I\!\!\:I\!\!\:I}}
\newcommand{\Scal}{\mathrm{Scal}}
\DeclareMathOperator{\Grad}{grad}
\DeclareMathOperator{\Div}{div}
\DeclareMathOperator{\Curl}{curl}

\begin{document}

\title[Complex-valued extension of mean curvature]{Complex-valued extension of mean curvature \\ for surfaces in Riemann-Cartan geometry}
\author[Dongha Lee]{Dongha Lee}
\address{Korea Advanced Institute of Science and Technology, 291, Daehak-ro, Yuseong-gu, Daejeon, Republic of Korea}
\email{leejydh97@kaist.ac.kr}
%\urladdr{\url{https://sites.google.com/view/dongha-lee/home}}
\date{\today}
\begin{abstract}
We extend the framework of submanifolds in Riemannian geometry to Riemann-Cartan geometry, which addresses connections with torsion. This procedure naturally introduces a 2-form on submanifolds associated with the nontrivial ambient torsion, whose Hodge dual plays the role of an imaginary counterpart to mean curvature for surfaces in a Riemann-Cartan 3-manifold. We observe that this complex-valued geometric quantity interacts with a number of other geometric concepts including the Hopf differential and the Gauss map, which generalizes classical minimal surface theory.
\end{abstract}
\maketitle

\section{Introduction}

Connections with torsion were first proposed and studied by Cartan~\cites{MR1509253, MR1509255, MR1509263}. His groundbreaking work had a wide influence on theoretical physics, especially on general relativity, which eventually led to what is now referred to as \emph{Einstein-Cartan theory} (see, e.g.,~\cite{TrautmanEC} for details). Moreover, in a different vein, Einstein~\cite{EinsteinTP} attempted to associate the torsion with the electromagnetic field to unify general relativity with electromagnetism, although this idea could not be fully realized. This approach has become known as \emph{teleparallelism}. For further historical background in physics, see~\cite{MR2037618}. Mathematically, Cartan's work has evolved into what is now called \emph{Riemann-Cartan geometry} (see, e.g.,~\cite{MR2866744}). It concerns a Riemannian (or pseudo-Riemannian) manifold endowed with a metric-compatible connection, which may have nonzero torsion. A trivial example is Riemannian geometry with the Levi-Civita connection. Another important example is \emph{Weitzenb{\"o}ck geometry} with a nonholonomic flat connection, which is based on a fixed global smooth frame (Definition~\ref{D:Weitzenbock-affine}).

Despite the profound influence of Cartan's work on theoretical physics, it seems that it has not received the same level of attention in mathematics, particularly in the study of submanifolds. This lack of study is surprising, given the potential of torsion in geometric and physical contexts. The present paper begins with extending the framework of submanifolds in Riemannian geometry to Riemann-Cartan geometry (of signature \( (n,0) \)), accounting for the presence of torsion. This requires us to carefully examine whether the fundamental geometric tools and concepts for submanifolds in Riemannian geometry work effectively in Riemann-Cartan geometry as well. We shall observe that this procedure naturally introduces a \( 2 \)-form on submanifolds associated with the nontrivial ambient torsion (Definition~\ref{D:torsion-form}).

For surfaces in a Riemann-Cartan \( 3 \)-manifold, it turns out that the Hodge dual of this \( 2 \)-form produces a smooth scalar field that plays the role of an ``imaginary counterpart'' to mean curvature. The main concern of the present paper is this complex-valued geometric quantity on surfaces, denoted by
\begin{align}
\boldsymbol{H} & = H + \boldsymbol{i} {\star} \tau,
\end{align}
where \( H \) is the mean curvature and \( \tau \) is the aforementioned \( 2 \)-form associated with the ambient torsion in Riemann-Cartan geometry (Definition~\ref{D:H}). This extends the notion of mean curvature in Riemannian geometry, as it descends to the usual mean curvature for the Levi-Civita connection. From the perspective of Weitzenb{\"o}ck geometry, this is an alternative way to extend the notion of mean curvature in Euclidean geometry, distinct from Riemannian geometry. The complex-valued quantity \( \boldsymbol{H} \) of our interest interacts with a number of other geometric concepts including the Hopf differential and the Gauss map. We shall present various results regarding this quantity, most of which generalize well-known fundamental results in classical minimal surface theory.

The outline of this paper is as follows. In Section~\ref{S:preliminaries}, we briefly introduce foundational concepts in Riemann-Cartan geometry and Weitzenb{\"o}ck geometry, which are necessary for the discussion in the present paper. In Section~\ref{S:submanifolds}, we delve into fundamental geometric tools such as the second fundamental form, the Weingarten map, and the extrinsic Gaussian and mean curvatures, which are adapted to reflect the presence of ambient torsion. It will be shown that most properties including the Gauss equation, Theorema Egregium, and the Gauss-Bonnet theorem work well in Riemann-Cartan geometry, despite the failure of some symmetries of the Riemann curvature (Propositions~\ref{P:SSFF},~\ref{P:TE}, and~\ref{P:Gauss-Bonnet}).

In Section~\ref{S:surfaces}, focusing on surfaces in a Riemann-Cartan 3-manifold, we explore the geometric properties of the complex-valued quantity \( \boldsymbol{H} \) in various perspectives. In Subsection~\ref{SS:divergence-curl}, we obtain a description of the quantity by the Gauss map in Weitzenb{\"o}ck geometry: the real part \( H \) is expressed as a certain ``divergence'' of the Gauss map, while the imaginary part \( {\star} \tau \) is expressed as a certain ``curl'' of the Gauss map (Proposition~\ref{P:divergence-curl} and Corollary~\ref{C:divergence-curl}). This agrees with the classical fact that the usual mean curvature equals the divergence of the Gauss map in Euclidean geometry. The following is one of the interesting observations in this subsection, regarding the choice of a global smooth frame.

\begin{pmain}[Theorem~\ref{T:gauge-transformation} and Corollary~\ref{C:gauge-transformation}] \label{M:A}
Let \( S \) be an oriented smooth surface embedded in an oriented parallelizable Riemannian\/ \( 3 \)-manifold \( M \). Let \( s \) be a global oriented orthonormal smooth frame for \( M \). Let \( n \colon S \to \mathbb{S}^2 \) be the Gauss map of \( S \) with respect to \( s \) (Definition~\ref{D:Gauss-map}). Then, for any smooth map \( \theta \colon S \to \mathbb{R} / 2 \boldsymbol{\pi} \mathbb{Z} \), the gauge transformation associated with the rotation \( \boldsymbol{e}^{\theta \hat{n}} \colon S \to \mathrm{SO}(3) \) of angle \( \theta \) about the axis \( n \) yields
\begin{align}
\boldsymbol{H} {\left( s \cdot \boldsymbol{e}^{\theta \hat{n}} \right)} & = \boldsymbol{H}(s) \boldsymbol{e}^{\boldsymbol{i} \theta}.
\end{align}
In particular, the gauge transformation of the Weitzenb{\"o}ck connection via a rotation about the normal axis does not change the norm\/ \( {\left| \boldsymbol{H} \right|} \).
\end{pmain}

\noindent More general gauge transformation formulas (i.e., for arbitrary rotations) are given in Lemma~\ref{L:gauge-transformation} and Proposition~\ref{P:gauge-transformation}, from which Theorem~\ref{M:A} follows.

In Subsection~\ref{SS:Hopf-differential}, we apply the general idea of the Hopf differential originated in~\cite{MR1013786} to the context of Riemann-Cartan geometry. Recall that an oriented smooth surface endowed with a Riemannian metric naturally inherits a unique complex structure from its conformal structure. We consider two quadratic differentials \( \varphi \) and \( \psi \) on surfaces defined by
\begin{align}
\varphi & = \II {\left( \frac{\partial}{\partial z}, \frac{\partial}{\partial z} \right)} \, \mathrm{d} z^2 \quad \text{and} \quad \psi = \III {\left( \frac{\partial}{\partial z}, \frac{\partial}{\partial z} \right)} \, \mathrm{d} z^2,
\end{align}
where \( \II \) and \( \III \) denote the second and third fundamental forms respectively in Riemann-Cartan geometry (Definitions~\ref{D:Hopf-differential} and~\ref{D:TFF}). In Riemannian geometry, the former quadratic differential \( \varphi \) corresponds to the Hopf differential. The main result of this subsection can be summarized as follows.

\begin{pmain}[Corollaries~\ref{C:Hopf-differential} and~\ref{C:TFF}] \label{M:B}
Let \( S \) be a smooth surface embedded in an oriented Riemann-Cartan\/ \( 3 \)-manifold, which is oriented by the unit normal vector field \( N \) along \( S \). Suppose that the following skew-symmetric smooth\/ \( (1,2) \)-tensor field \( L \) on \( S \) vanishes everywhere.
\begin{align}
L {\left( X, Y \right)} & = \tilde{R} {\left( X, Y \right)} N - J W J T {\left( X, Y \right)} \qquad (X, Y \in \Gamma(TS)),
\end{align}
where \( \tilde{R} \) is the ambient curvature, \( J \) is the almost complex structure on \( S \), \( W \) is the Weingarten map, and \( T \) is the torsion of \( S \). Then the following two hold.
\begin{enumerate}
\item \( \boldsymbol{H} \) is a holomorphic function on \( S \) iff \( \varphi \) is a holomorphic quadratic differential on \( S \).
\item If \( \boldsymbol{H} \) is a holomorphic function on \( S \), then \( \psi \) is a holomorphic quadratic differential on \( S \).
\end{enumerate}
\end{pmain}

\noindent The part~(a) of Theorem~\ref{M:B} generalizes the classical characterization by Hopf that a surface has constant mean curvature exactly when the Hopf differential becomes a holomorphic quadratic differential (Example~\ref{Ex:Einstein}).

In Subsection~\ref{SS:minimal}, we discuss minimal surfaces and totally umbilic surfaces in Riemann-Cartan geometry (Definition~\ref{D:minimal}). First of all, the following explains what the quadratic differentials \( \varphi \) and \( \psi \) measure.

\begin{pmain}[Propositions~\ref{P:umbilic} and~\ref{P:TFF-minimal}] \label{M:C}
Let \( S \) be an oriented smooth surface embedded in an oriented Riemann-Cartan\/ \( 3 \)-manifold. Then the following two hold.
\begin{enumerate}
\item \( \varphi \) vanishes everywhere iff \( S \) is totally umbilic.
\item If \( S \) is connected and nowhere geodesic, then \( \psi \) vanishes everywhere iff \( S \) is minimal (i.e., \( \boldsymbol{H} = 0 \) everywhere) or totally umbilic.
\end{enumerate}
\end{pmain}

\noindent Via this observation, we capture the fact that the conformality of the Gauss map is deeply related to the minimality of surface as follows.

\begin{pmain}[Corollary~\ref{C:main}] \label{M:D}
Let \( S \) be an oriented smooth surface embedded in a Weitzenb{\"o}ck\/ \( 3 \)-manifold \( M \) with global smooth frame \( s \) (Definition~\ref{D:Weitzenbock-affine}). Then the following two hold.
\begin{enumerate}
\item (Local version) The Gauss map \( n^s \colon S \to \mathbb{S}^2 \) is conformal at a point in \( S \) iff\/ \( \II^s \ne 0 \) at the point and either \( \boldsymbol{H}^s = 0 \) or \( S \) is umbilic at the point.
\item (Global version) If \( S \) is connected, then the Gauss map \( n^s \colon S \to \mathbb{S}^2 \) is conformal iff \( S \) is nowhere geodesic and either it is minimal or totally umbilic.
\end{enumerate}
\end{pmain}

\noindent There is a more general version of Theorem~\ref{M:D} for flat Riemann-Cartan geometry in Theorem~\ref{T:main}. This result generalizes the classical fact in minimal surface theory that the conformality of the Gauss map characterizes minimal surfaces in \( \mathbb{R}^3 \) (Corollary~\ref{C:classical}). There have been several notable attempts to extend this classical fact in different contexts as well, e.g.,~\cite{MR0955072},~\cite{MR1873007}, and~\cite{MR3311902}.

\subsubsection*{Main contribution}

The main contribution of this work is to introduce the complex-valued quantity \( \boldsymbol{H} = H + \boldsymbol{i} {\star} \tau \) for surfaces in Riemann-Cartan geometry, which generalizes mean curvature and relevant theorems in minimal surface theory. This geometric quantity arises naturally in the process of extending the framework of submanifolds in Riemannian geometry to Riemann-Cartan geometry.

An important aspect of mean curvature in Riemannian geometry is that it naturally emerges from the variation of the area functional. A limitation of the present paper is that it does not provide such a variational interpretation---namely, whether the quantity \( \boldsymbol{H} \) arises from the variation of a certain geometric invariant---although the term ``minimal'' is used throughout the paper as a generalization.

However, the significance of the work lies in presenting a broader or distinct perspective that generalizes the classical notion of mean curvature. The quantity \( \boldsymbol{H} \) is defined in Riemann-Cartan geometry, which encompasses both Riemannian and Weitzenb{\"o}ck geometries (Figure~\ref{FIG:geometry}). For example, Euclidean geometry admits multiple characterizations of the minimality. On the one hand, Riemannian geometry can generalize the characterization based on the extremality of the area functional, while it does not easily generalize theorems involving the Gauss map. On the other hand, Weitzenb{\"o}ck geometry can generalize the characterization based on the conformality of the Gauss map, as in Theorem~\ref{M:D}\@.

\begin{figure}
\begin{tikzpicture}
\node[align=center] at (0,4) (RC) {Riemann-Cartan geometry \\ (\( R \ne 0 \), \( T \ne 0 \)) \\[.15cm] Theorems~\ref{M:B},~\ref{M:C},~\( \ldots \)};
\node[align=center] at (-3,2) (R) {Riemannian geometry \\ (\( R \ne 0 \), \( T = 0 \)) \\[.15cm] Minimal surface theory};
\node[align=center] at (3,2) (W) {Weitzenb{\"o}ck geometry \\ (\( R = 0 \), \( T \ne 0 \)) \\[.15cm] Theorems~\ref{M:A},~\ref{M:D},~\( \ldots \)};
\node[align=center] at (0,0) (E) {Euclidean geometry \\ (\( R = 0 \), \( T = 0 \)) \\[.15cm] Minimal surface theory};
\draw (RC) -- (R);
\draw (RC) -- (W);
\draw (R) -- (E);
\draw (W) -- (E);
\draw (-2.5,3.77) -- (2.5,3.77) -- (2.5,4.77) -- (-2.5,4.77) -- cycle;
\draw (-5.5,1.77) -- (-.5,1.77) -- (-.5,2.77) -- (-5.5,2.77) -- cycle;
\draw (.5,1.77) -- (5.5,1.77) -- (5.5,2.77) -- (.5,2.77) -- cycle;
\draw (-2.5,-.23) -- (2.5,-.23) -- (2.5,.77) -- (-2.5,.77) -- cycle;
\end{tikzpicture}
\caption{Four geometries. Riemannian and Weitzenb{\"o}ck geometries are two distinct ways to generalize Euclidean geometry.}
\label{FIG:geometry}
\end{figure}

Meanwhile, the quantity \( \boldsymbol{H} \) was first discovered in the context of the renormalization of the Chern-Simons invariant for noncompact hyperbolic \( 3 \)-manifolds having infinite volume. The leading coefficient of the asymptotics arising from this procedure corresponds to the integral of \( \boldsymbol{H} \). Therefore, studying the quantity \( \boldsymbol{H} \) may also contribute to a better understanding of the renormalized Chern-Simons invariant for such hyperbolic \( 3 \)-manifolds. For details, see~\cite{LeeCS}.

\subsubsection*{Notations and conventions}

This paper only considers Riemannian metrics, not pseudo-Riemannian metrics. All Riemann-Cartan manifolds in this paper are assumed to have signature \( (n,0) \). All orientations obey the right-hand rule. A smooth hypersurface \( S \) embedded in an oriented Riemannian manifold \( M \) is said to be oriented by a unit normal vector field \( N \) along \( S \) iff a frame \( {\left( E_1, \dotsc, E_{n-1} \right)} \) for \( S \) is oriented exactly when the frame \( {\left( N, E_1, \dotsc, E_{n-1} \right)} \) for \( M \) is oriented.

We will often make use of the Einstein summation convention to achieve brevity: An index variable appearing twice in a single term (once in an upper position and once in a lower position) assumes the summation of the term over the index. For example, \( A^i B_i = \sum_i A^i B_i \). The ``cyclic'' summation is denoted as follows.
\begin{align}
\sum_{\mathrm{cyc}} A_{ijk} & = \sum_{\substack{i, j, k \\ \mathrm{cyc}}} A_{ijk} = A_{123} + A_{231} + A_{312}.
\end{align}
The \( (i,j) \)-entry of a matrix (or a matrix-valued object) \( A \) is denoted by \( A^i_j \).

\section{Preliminaries} \label{S:preliminaries}

This section provides foundational concepts that are necessary for the present paper. We give a brief introduction to Riemann-Cartan geometry and Weitzenb{\"o}ck geometry, including essential definitions, notations, and properties which will be used throughout this paper. It should be noted that this paper only considers Riemann-Cartan manifolds of signature \( (n,0) \) (i.e., with Riemannian metric), whereas the literature on Riemann-Cartan geometry typically allows pseudo-Riemannian metrics for Riemann-Cartan manifolds.

\subsection{Riemann-Cartan geometry}

We start with providing several basic definitions regarding connection on a smooth manifold, which also introduce the notations and conventions that the present paper follows. They can be found in differential geometry literature, e.g.,~\cite{MR3887684} and~\cite{MR3585539}.

\begin{pdef} \label{D:connection}
A \emph{(Koszul) connection} on a smooth real vector bundle \( E \surjto M \) is an \( \mathbb{R} \)-bilinear map \( \del \colon \Gamma(TM) \times \Gamma(E) \to \Gamma(E) \) denoted by \( {\left( X, \xi \right)} \mapsto \del_X \xi \) such that
\begin{align}
\del_{f X} \xi & = f \del_X \xi \quad \text{and} \quad \del_X {\left( f \xi \right)} = X(f) \xi + f \del_X \xi
\end{align}
for all \( f \in \mathcal{C}^{\infty} {\left( M, \mathbb{R} \right)} \), \( X \in \Gamma(TM) \), and \( \xi \in \Gamma(E) \). A \emph{(affine) connection} on a smooth manifold \( M \) means a connection on the tangent bundle \( TM \surjto M \).
\end{pdef}

\noindent In Definition~\ref{D:connection}, it is an elementary fact that the value of \( \del_X \xi \) at \( x \in M \) depends only on the value of \( X \) at \( x \) and the local behavior of \( \xi \) near \( x \).

\begin{pdef}
Let \( \del \) be a connection on a smooth manifold \( M \).
\begin{enumerate}
\item The \emph{torsion} of \( \del \) is the smooth \( (1,2) \)-tensor field \( T \) on \( M \) defined by
\begin{align}
T(X,Y) & = \del_X Y - \del_Y X - [X, Y],
\end{align}
where \( X, Y \in \Gamma(TM) \). \( \del \) is said to be \emph{torsion-free} iff \( T = 0 \) everywhere.
\item The \emph{curvature} of \( \del \) is the smooth \( (1,3) \)-tensor field \( R \) on \( M \) defined by
\begin{align}
R(X,Y) Z & = \del_X \del_Y Z - \del_Y \del_X Z - \del_{[X, Y]} Z,
\end{align}
where \( X, Y, Z \in \Gamma(TM) \). \( \del \) is said to be \emph{flat} iff \( R = 0 \) everywhere.
\item The \emph{Ricci curvature} of \( \del \) is the smooth \( (0,2) \)-tensor field \( \Ric \) on \( M \) given by the trace of the linear map \( Z \mapsto R(Z,X) Y \), i.e.,
\begin{align}
\Ric(X,Y) & = \tr \! {\left( Z \mapsto R(Z,X) Y \right)},
\end{align}
where \( X, Y, Z \in \Gamma(TM) \). \( \del \) is said to be \emph{Ricci-flat} iff \( \Ric = 0 \) everywhere.
\end{enumerate}
\end{pdef}

\noindent The Riemannian metric on a Riemannian manifold will be denoted by \( {\left\langle {}\cdot{}, {}\cdot{} \right\rangle} \).

\begin{pdef}
A connection \( \del \) on a Riemannian manifold \( M \) is said to be \emph{metric-compatible} iff \( X {\left\langle Y, Z \right\rangle} = {\left\langle \del_X Y, Z \right\rangle} + {\left\langle Y, \del_X Z \right\rangle} \) for all \( X, Y, Z \in \Gamma(TM) \).
\end{pdef}

It is a fundamental theorem in Riemannian geometry that every Riemannian manifold has a unique torsion-free metric-compatible connection, called the \emph{Levi-Civita connection}. A \emph{Riemann-Cartan manifold} is a Riemannian manifold endowed with a metric-compatible connection, which is not necessarily torsion-free. It also has the Levi-Civita connection as a Riemannian manifold, which will be denoted by \( \bar{\del} \). Quantities related to the Levi-Civita connection will be denoted by a bar.

\begin{pdef}
Let \( M \) be a Riemann-Cartan manifold.
\begin{enumerate}
\item The \emph{Riemann curvature} of \( M \) is the smooth \( (0,4) \)-tensor field \( R \) on \( M \) defined by \( R {\left( W, X, Y, Z \right)} = {\left\langle R {\left( W, X \right)} Y, Z \right\rangle} \), where \( W, X, Y, Z \in \Gamma(TM) \). This is an abuse of notation.
\item The \emph{scalar curvature} of \( M \) is the smooth scalar field \( \Scal \) on \( M \) defined by the trace of the Ricci curvature of \( M \) with respect to the metric.
\end{enumerate}
\end{pdef}

\noindent The Riemann curvature of a Riemann-Cartan manifold satisfies only some of the symmetries that hold in the case of a Riemannian manifold as follows, due to the presence of nontrivial torsion.

\begin{pprop}
The Riemann curvature \( R \) of a Riemann-Cartan manifold \( M \) satisfies \( R {\left( W, X, Y, Z \right)} = - R {\left( X, W, Y, Z \right)} \) and \( R {\left( W, X, Y, Z \right)} = - R {\left( W, X, Z, Y \right)} \) for all \( W, X, Y, Z \in \Gamma(TM) \).
\end{pprop}
\begin{proof}
The first equality is obvious by definition. For the second equality, first observe that \( R {\left( W, X, Y, Y \right)} = 0 \) for all \( W, X, Y \in \Gamma(TM) \), since
\begin{align*}
W X {\left\langle Y, Y \right\rangle} & = W {\left( 2 {\left\langle \del_X Y, Y \right\rangle} \right)} = 2 {\left\langle \del_W \del_X Y, Y \right\rangle} + 2 {\left\langle \del_X Y, \del_W Y \right\rangle}, \\
X W {\left\langle Y, Y \right\rangle} & = X {\left( 2 {\left\langle \del_W Y, Y \right\rangle} \right)} = 2 {\left\langle \del_X \del_W Y, Y \right\rangle} + 2 {\left\langle \del_W Y, \del_X Y \right\rangle}, \\
[W, X] {\left\langle Y, Y \right\rangle} & = 2 {\left\langle \del_{[W, X]} Y, Y \right\rangle}
\end{align*}
for all \( W, X, Y \in \Gamma(TM) \). Then the second equality follows from the expansion of the equality \( R {\left( W, X, Y+Z, Y+Z \right)} = 0 \), where \( W, X, Y, Z \in \Gamma(TM) \).
\end{proof}

There is another useful characterization of Riemann-Cartan manifold, which uses the notion of frame bundle and principal bundle. We give a brief account.

\begin{pdef}
Let \( G \) be a Lie group. A \emph{smooth principal \( G \)-bundle} is a smooth fiber bundle \( \pi \colon P \surjto M \) endowed with a smooth free right group action of \( G \) on \( P \) such that every point in \( M \) has an \emph{equivariant local trivialization}, i.e., a local trivialization \( \varphi \) satisfying \( \varphi(p \cdot g) = \varphi(p) \cdot g \) for all \( p \) in the fibers and \( g \in G \). We define a map \( \iota_p \colon G \to {\left. P \right|}_{\pi(p)} \) by \( \iota_p(g) = p \cdot g \) for each \( p \in P \), and a map \( R_g \colon P \to P \) by \( R_g(p) = p \cdot g \) for each \( g \in G \), which are diffeomorphisms.
\end{pdef}

\begin{pdef}
Let \( G \) be a Lie group with the Lie algebra \( \mathfrak{g} \) and the Maurer-Cartan\/ \( 1 \)-form \( \mu \). A \emph{(principal) connection} on a smooth principal \( G \)-bundle \( \pi \colon P \surjto M \) is \( \omega \in \Omega^1 {\left( P, \mathfrak{g} \right)} \) such that \( \iota_p^* \omega = \mu \) for all \( p \in P \) and \( R_g^* \omega = \Ad_{g^{-1}} \omega \) for all \( g \in G \). Its \emph{curvature} is defined by \( \Omega = \mathrm{d} \omega + \frac{1}{2} {\left[ \omega \wedge \omega \right]} \in \Omega^2 {\left( P, \mathfrak{g} \right)} \).
\end{pdef}

\begin{plem} \label{L:gauge-transformation-connection}
Let \( G \) be a Lie group. Let \( \pi \colon P \surjto M \) be a smooth principal \( G \)-bundle with a connection \( \omega \). Let\/ \( \Omega \) be the curvature of \( \omega \). If \( \varphi \colon P \isoto P \) is a gauge transformation with the associated map \( g_{\varphi} \colon P \to G \) given by \( \varphi(p) = p \cdot g_{\varphi}(p) \) for all \( p \in P \), then
\begin{align}
\varphi^* \omega & = \Ad_{g_{\varphi}^{-1}} \omega + g_{\varphi}^* \mu \quad \text{and} \quad \varphi^* \Omega = \Ad_{g_{\varphi}^{-1}} \Omega,
\end{align}
where \( \mu \) is the Maurer-Cartan\/ \( 1 \)-form on \( G \).
\end{plem}

Given a Riemannian \( n \)-manifold \( M \), its orthonormal frame bundle \( FM \surjto M \) is a smooth principal \( \mathrm{O}(n) \)-bundle, and there is a well-known one-to-one correspondence between two notions of connection as follows.
\begin{align}
\begin{aligned}
& {\left\{ \text{metric-compatible Koszul connections \( \del \) on \( TM \surjto M \)} \right\}} \\
& \qquad \longbijto {\left\{ \text{principal \( \mathrm{O}(n) \)-connections \( \omega \) on \( FM \surjto M \)} \right\}}.
\end{aligned} \label{E:connection}
\end{align}
The bridge is given by \emph{connection coefficients}: If a principal connection \( \omega \) is given, then the corresponding Koszul connection \( \del \) has the connection coefficients \( s^* \omega^i_j \) with respect to any local orthonormal smooth frame \( s \) for \( M \). Therefore, a Riemann-Cartan manifold can be understood as the choice of a principal connection on the orthonormal frame bundle over a Riemannian manifold. See Chapter~6 of~\cite{MR3585539} for further details. For torsion, we need to introduce one more notion.

\begin{pdef}
Let \( M \) be a Riemannian \( n \)-manifold. The \emph{solder form} of the orthonormal frame bundle \( \pi \colon FM \surjto M \) is \( \theta \in \Omega^1 {\left( FM, \mathbb{R}^n \right)} \) defined by
\begin{align}
{\left. \theta \right|}_{(x,f)} (X) & = f^{-1} {\left( \pi_* X \right)} \qquad ((x,f) \in FM, \ X \in T_{(x,f)} FM),
\end{align}
where \( f \) is understood as the map \( f \colon \mathbb{R}^n \isoto T_x M \). The \emph{torsion} of a connection \( \omega \in \Omega^1 {\left( FM, \mathfrak{o}(n) \right)} \) is defined by \( \Theta = \mathrm{d} \theta + \omega \wedge \theta \in \Omega^2 {\left( FM, \mathbb{R}^n \right)} \).
\end{pdef}

A Riemann-Cartan manifold has four fundamental forms: the solder form \( \theta \), the connection \( \omega \), the torsion \( \Theta \), and the curvature \( \Omega \). If \( s = {\left( E_1, \dotsc, E_n \right)} \) is a local orthonormal smooth frame, then they satisfy
\begin{align}
\begin{aligned}
X & = s^* \theta^i(X) E_i, & \quad T(X,Y) & = s^* \Theta^i(X,Y) E_i, \\
\del_X E_j & = s^* \omega^i_j(X) E_i, & \quad R(X,Y) E_j & = s^* \Omega^i_j(X,Y) E_i
\end{aligned}
\end{align}
for any vectors \( X \) and \( Y \) (at the same point), where the Einstein summation convention is assumed. Also, they satisfy
\begin{align}
\begin{aligned}
\mathrm{d} \theta + \omega \wedge \theta & = \Theta, & \quad \mathrm{d} \Theta + \omega \wedge \Theta & = \Omega \wedge \theta, \\
\mathrm{d} \omega + \omega \wedge \omega & = \Omega, & \quad \mathrm{d} \Omega + \omega \wedge \Omega & = \Omega \wedge \omega.
\end{aligned} \label{E:fundamental-equations}
\end{align}
These are four fundamental equations.

\subsection{Weitzenb{\"o}ck geometry}

A trivial example of Riemann-Cartan geometry is Riemannian geometry with the Levi-Civita connection, which has \emph{zero torsion} (and possibly nonzero curvature). Another important example of Riemann-Cartan geometry is Weitzenb{\"o}ck geometry, which has \emph{zero curvature} (and possibly nonzero torsion). A modern definition can be found, e.g., in Section~5.7 of~\cite{MR0615912}.

\begin{pdef} \label{D:Weitzenbock-affine}
A \emph{Weitzenb{\"o}ck manifold} is a parallelizable smooth manifold endowed with a global smooth frame. If \( M \) is a Weitzenb{\"o}ck \( n \)-manifold with global smooth frame \( s = {\left( E_1, \dotsc, E_n \right)} \), it automatically possesses the oriented Riemann-Cartan structure as follows.
\begin{enumerate}
\item The orientation on \( M \) is determined by \( s = {\left( E_1, \dotsc, E_n \right)} \).
\item The Riemannian metric on \( M \) is determined by \( {\left\langle E_i, E_j \right\rangle} = \delta_{ij} \).
\item The metric-compatible connection on \( M \), called the \emph{Weitzenb{\"o}ck connection}, is defined by
\begin{align}
\del^s_X Y & = \sum_{i = 1}^n X {\left( Y^i \right)} E_i,
\end{align}
where \( X, Y \in \Gamma(TM) \) and \( Y = Y^1 E_1 + \dotsb + Y^n E_n \).
\end{enumerate}
\end{pdef}

\noindent Note that the Weitzenb{\"o}ck structure on the Euclidean space \( \mathbb{R}^n \) given by the standard global smooth frame \( {\big( \frac{\partial}{\partial x^1}, \dotsc, \frac{\partial}{\partial x^n} \big)} \) and the Riemannian structure given by the standard inner product induce the same Riemann-Cartan structure. From this perspective, Riemannian geometry and Weitzenb{\"o}ck geometry are two distinct approaches to the generalization of Euclidean geometry.

Some basic properties of the Weitzenb{\"o}ck connection are shown in the following.

\begin{pprop} \label{P:Weitzenbock}
Let \( M \) be a Weitzenb{\"o}ck \( n \)-manifold with global smooth frame \( s = {\left( E_1, \dotsc, E_n \right)} \). Then the following three hold.
\begin{enumerate}
\item The connection coefficients of\/ \( \del^s \) with respect to \( s \) are all zero.
\item The torsion \( T^s \) of\/ \( \del^s \) is given by
\begin{align}
T^s(X,Y) & = - \sum_{i, j = 1}^n X^i Y^j {\left[ E_i, E_j \right]}, \quad \text{i.e.,} \quad T^s {\left( E_i, E_j \right)} = - {\left[ E_i, E_j \right]}
\end{align}
for all \( X, Y \in \Gamma(TM) \), where \( X = \sum_{i = 1}^n X^i E_i \) and \( Y = \sum_{j = 1}^n Y^j E_j \).
\item \( \del^s \) is flat.
\end{enumerate}
\end{pprop}
\begin{proof}
For all \( X, Y, Z \in \Gamma(TM) \), we have \( \del^s_X E_j = X(1) E_j = 0 \),
\begin{align}
T^s(X,Y) & = X {\left( Y^i \right)} E_i - Y {\left( X^i \right)} E_i - {\left[ X^i E_i, Y^j E_j \right]} = - X^i Y^j {\left[ E_i, E_j \right]}, \\
R^s(X,Y) Z & = X Y {\left( Z^i \right)} E_i - Y X {\left( Z^i \right)} E_i - {\left[ X, Y \right]} {\left( Z^i \right)} E_i = 0,
\end{align}
where the Einstein summation convention is assumed.
\end{proof}

From the viewpoint of frame bundle and principal bundle, the Weitzenb{\"o}ck connection can be understood as the choice of a trivial connection via a global smooth section in the following sense.

\begin{pdef}
Let \( G \) be a Lie group. Let \( M \) be a smooth manifold. The \emph{trivial principal \( G \)-bundle} over \( M \) is defined by the projection \( M \times G \surjto M \) with the smooth right group action \( (x, g) \cdot h = (x, gh) \), where \( (x, g) \in M \times G \) and \( h \in G \). The \emph{trivial connection} on this is the pullback of the Maurer-Cartan \( 1 \)-form on \( G \) via the projection \( M \times G \surjto G \).
\end{pdef}

\begin{pdef} \label{D:Weitzenbock-principal}
Let \( G \) be a Lie group. Let \( \pi \colon P \surjto M \) be a smooth principal \( G \)-bundle. Suppose that it has a global smooth section \( \sigma \colon M \injto P \). Then it induces a global trivialization \( \varphi_{\sigma} \colon P \to M \times G \) defined by \( \varphi_{\sigma}^{-1}(x, g) = \sigma(x) \cdot g \).
\begin{align}
\begin{tikzcd}[ampersand replacement=\&]
P \ar[rr, "\varphi_{\sigma}"] \ar[rd, two heads, "\pi"'] \& \& M \times G \ar[ld, two heads] \\
\& M
\end{tikzcd}
\end{align}
The \emph{Weitzenb{\"o}ck connection} on \( \pi \colon P \surjto M \) induced by \( \sigma \) is the principal connection \( \omega^{\sigma} \in \Omega^1 {\left( P, \mathfrak{g} \right)} \) on \( \pi \colon P \surjto M \) defined by the pullback of the trivial connection on the trivial principal \( G \)-bundle \( M \times G \surjto M \) via \( \varphi_{\sigma} \colon P \to M \times G \).
\end{pdef}

\noindent Weitzenb{\"o}ck connections of Definition~\ref{D:Weitzenbock-principal} are flat (i.e., \( \Omega^{\sigma} = 0 \)) as well, by the Maurer-Cartan equation. In fact, it is easy to check that \( \sigma^* \omega^{\sigma} = 0 \).

Definitions~\ref{D:Weitzenbock-affine} and~\ref{D:Weitzenbock-principal} indeed coincide in the following sense.

\begin{pprop}
Let \( M \) be a Weitzenb{\"o}ck \( n \)-manifold with global smooth frame \( s \). Then, via the one-to-one correspondence~\eqref{E:connection}, the Weitzenb{\"o}ck affine connection\/ \( \del^s \) corresponds to the Weitzenb{\"o}ck principal\/ \( \mathrm{O}(n) \)-connection \( \omega^s \).
\end{pprop}
\begin{proof}
It is straightforward to check that \( s^* \omega^s = 0 \). Indeed, the connection coefficients of \( \del^s \) with respect to \( s \) are all zero by Proposition~\ref{P:Weitzenbock}.
\end{proof}

\begin{prmrk}
A flat connection is not necessarily a Weitzenb{\"o}ck connection. The difference between them is \emph{holonomy}. Weitzenb{\"o}ck connections are globally nonholonomic, while flat connections can be holonomic. See Sections~3.3--3.10 in~\cite{HimpelCS}.
\end{prmrk}

\begin{prmrk}
All Weitzenb{\"o}ck connections are \emph{gauge-equivalent} to each other. In the setup of Definition~\ref{D:Weitzenbock-principal}, suppose that \( \sigma, \sigma' \colon M \injto P \) are two global smooth sections. Then \( \omega^{\sigma'} = \varphi^* \omega^{\sigma} \) by the gauge transformation \( \varphi = \varphi_{\sigma}^{-1} \circ \varphi_{\sigma'} \colon P \isoto P \).
\end{prmrk}

\section{Riemann-Cartan submanifolds} \label{S:submanifolds}

In this section, we extend the framework of submanifolds in Riemannian geometry to Riemann-Cartan geometry. We consider fundamental geometric tools such as the second fundamental form, the Weingarten map, and the extrinsic Gaussian and mean curvatures, which are adapted to account for the presence of ambient torsion. Readers are encouraged to pay attention to the differences from Riemannian geometry. Throughout this section, quantities related to the ambient manifold will be denoted by a tilde. For example, \( \tilde{\del} \), \( \tilde{T} \), \( \tilde{R} \), etc.

\subsection{The second fundamental form}

A Riemannian submanifold has the induced metric given by the restriction of the metric on the ambient manifold. A connection on the ambient manifold induces a connection via the projection.

\begin{plem} \label{L:induced-connection}
Let \( M \) be a Riemannian submanifold embedded in a Riemannian manifold \( \tilde{M} \). If\/ \( \tilde{\del} \) is a connection on \( \tilde{M} \), then the following produces a well-defined connection\/ \( \del \) on \( M \).
\begin{align}
\del_X Y & = {\big( \tilde{\del}_X \tilde{Y} \big)}{}^{\top},
\end{align}
where \( X, Y \in \Gamma(TM) \) and \( \tilde{Y} \) is an arbitrary smooth extension of \( Y \) to an open subset of \( \tilde{M} \). Furthermore, the following three hold.
\begin{enumerate}
\item If\/ \( \tilde{\del} \) is torsion-free, then so is\/ \( \del \).
\item If\/ \( \tilde{\del} \) is metric-compatible, then so is\/ \( \del \).
\item If\/ \( \tilde{\del} \) is the Levi-Civita connection, then so is\/ \( \del \).
\end{enumerate}
\end{plem}
\begin{proof}
See, e.g., Chapter~8 of~\cite{MR3887684}.
\end{proof}

\noindent It is natural to define a \emph{Riemann-Cartan submanifold} embedded in a Riemann-Cartan manifold as a Riemannian submanifold with the metric-compatible connection defined by Lemma~\ref{L:induced-connection}, which is a Riemann-Cartan manifold. Throughout this paper, we will assume without mentioning that every smooth submanifold embedded in a Riemann-Cartan manifold inherits the Riemannian metric and the metric-compatible connection in this manner.

The second fundamental form in Riemann-Cartan geometry is defined as follows.

\begin{pdef} \label{D:SFF}
Let \( M \) be a smooth submanifold embedded in a Riemann-Cartan manifold \( \tilde{M} \). Let \( NM \surjto M \) be the normal bundle. The \emph{second fundamental form} of \( M \) is the map \( \hat{\II} \colon \Gamma(TM) \times \Gamma(TM) \to \Gamma(NM) \) defined by
\begin{align}
\hat{\II}(X,Y) & = {\big( \tilde{\del}_X \tilde{Y} \big)}{}^{\perp},
\end{align}
where \( X, Y \in \Gamma(TM) \) and \( \tilde{Y} \) is an arbitrary smooth extension of \( Y \) to an open subset of \( \tilde{M} \). \( M \) is said to be \emph{totally geodesic} in \( \tilde{M} \) iff \( \hat{\II} = 0 \) everywhere.
\end{pdef}

\begin{plem} \label{L:SFF}
In the setup of Definition~\ref{D:SFF}, the following four hold.
\begin{enumerate}
\item \( \hat{\II} \) is well defined, i.e.,\/ \( \hat{\II}(X,Y) \) does not depend on the choice of a smooth extension of \( Y \), where \( X, Y \in \Gamma(TM) \).
\item \( \hat{\II} \) is \( \mathcal{C}^{\infty} {\left( M, \mathbb{R} \right)} \)-bilinear.
\item \( \hat{\II}(X,Y) - \hat{\II}(Y,X) = \tilde{T}(X,Y)^{\perp} \) for all \( X, Y \in \Gamma(TM) \).
\item The value of\/ \( \hat{\II}(X,Y) \) at \( x \in M \) depends only on \( {\left. X \right|}_x \) and \( {\left. Y \right|}_x \).
\end{enumerate}
\end{plem}
\begin{proof}
The equality in~(c) holds for any smooth extensions of \( X, Y \in \Gamma(TM) \). This implies~(d), and then the others also follow.
\end{proof}

\noindent Note that the second fundamental form in Riemann-Cartan geometry fails to be symmetric in general. It depends on the normal component of ambient torsion.

The Gauss equation holds in Riemann-Cartan geometry as well.

\begin{pprop} \label{P:SFF}
Let \( M \) be a smooth submanifold embedded in a Riemann-Cartan manifold \( \tilde{M} \). For all \( W, X, Y, Z \in \Gamma(TM) \), the following three hold.
\begin{gather}
\tilde{\del}_X Y = \del_X Y + \hat{\II}(X,Y), \qquad \tilde{T}(X,Y) = T(X,Y) + \hat{\II}(X,Y) - \hat{\II}(Y,X), \\
\tilde{R}(W,X,Y,Z) = R(W,X,Y,Z) - {\big\langle \hat{\II}(W,Z), \hat{\II}(X,Y) \big\rangle} + {\big\langle \hat{\II}(W,Y), \hat{\II}(X,Z) \big\rangle}.
\end{gather}
\end{pprop}
\begin{proof}
The only nontrivial equality is the last one. Fix \( W, X, Y, Z \in \Gamma(TM) \) and extend them arbitrarily smoothly to an open subset of \( \tilde{M} \). Along \( M \), we have
\begin{align*}
& \tilde{R}(W,X,Y,Z) = {\big\langle \tilde{\del}_W \tilde{\del}_X Y - \tilde{\del}_X \tilde{\del}_W Y - \tilde{\del}_{[W, X]} Y, Z \big\rangle} \\
& \qquad = {\big\langle \tilde{\del}_W {\big( \del_X Y + \hat{\II}(X,Y) \big)} - \tilde{\del}_X {\big( \del_W Y + \hat{\II}(W,Y) \big)} - \tilde{\del}_{[W, X]} Y, Z \big\rangle} \\
& \qquad = \begin{aligned}[t]
& {\big\langle \tilde{\del}_W \del_X Y - \tilde{\del}_X \del_W Y - \tilde{\del}_{[W, X]} Y, Z \big\rangle} \\
& + {\big\langle \tilde{\del}_W \hat{\II}(X,Y), Z \big\rangle} - {\big\langle \tilde{\del}_X \hat{\II}(W,Y), Z \big\rangle}
\end{aligned} \\
& \qquad = \begin{aligned}[t]
& {\left\langle \del_W \del_X Y - \del_X \del_W Y - \del_{[W, X]} Y, Z \right\rangle} \\
& - {\big\langle \hat{\II}(X,Y), \tilde{\del}_W Z \big\rangle} + {\big\langle \hat{\II}(W,Y), \tilde{\del}_X Z \big\rangle}
\end{aligned} \\
& \qquad = R(W,X,Y,Z) - {\big\langle \hat{\II}(X,Y), \hat{\II}(W,Z) \big\rangle} + {\big\langle \hat{\II}(W,Y), \hat{\II}(X,Z) \big\rangle}.
\end{align*}
This completes the proof.
\end{proof}

Now, we restrict ourselves to the case of oriented smooth hypersurfaces. In this case, the orientation determines a unique normal direction.

\begin{pdef} \label{D:SSFF}
Let \( M \) be an oriented smooth hypersurface embedded in an oriented Riemann-Cartan manifold \( \tilde{M} \), where \( M \) is oriented by the unit normal vector field \( N \) along \( M \). The \emph{(scalar) second fundamental form} of \( M \) is the smooth \( (0,2) \)-tensor field \( \II \) on \( M \) defined by
\begin{align}
\II(X,Y) & = {\big\langle N, \hat{\II}(X,Y) \big\rangle}, \quad \text{i.e.,} \quad \hat{\II}(X,Y) = \II(X,Y) N,
\end{align}
where \( X, Y \in \Gamma(TM) \).
\end{pdef}

\noindent Again, this fails to be symmetric in general. Indeed, it is symmetric iff the following differential \( 2 \)-form vanishes.

\begin{pdef}[Torsion \( 2 \)-form] \label{D:torsion-form}
In the setup of Definition~\ref{D:SSFF}, the \emph{torsion\/ \( 2 \)-form} of \( M \) is \( \tau \in \Omega^2 {\left( M, \mathbb{R} \right)} \) defined by
\begin{align}
\tau(X,Y) & = {\big\langle N, \tilde{T}(X,Y) \big\rangle},
\end{align}
where \( X, Y \in \Gamma(TM) \).
\end{pdef}

\begin{pprop} \label{P:SSFF}
In the setup of Definition~\ref{D:SSFF}, for all \( W, X, Y, Z \in \Gamma(TM) \), the following four hold.
\begin{gather}
\tilde{\del}_X Y = \del_X Y + \II(X,Y) N, \qquad \II(X,Y) - \II(Y,X) = \tau(X,Y), \\
\tilde{T}(X,Y) = T(X,Y) + \tau(X,Y) N, \\
\tilde{R}(W,X,Y,Z) = R(W,X,Y,Z) - \II(W,Z) \, \II(X,Y) + \II(W,Y) \, \II(X,Z).
\end{gather}
\end{pprop}
\begin{proof}
This is an immediate corollary of Lemma~\ref{L:SFF} and Proposition~\ref{P:SFF}.
\end{proof}

The Weingarten map in Riemann-Cartan geometry is given as follows.

\begin{pdef} \label{D:Weingarten-map}
Suppose the setup of Definition~\ref{D:SSFF}. The \emph{Weingarten map} of \( M \) is the smooth \( (1,1) \)-tensor field \( W \) on \( M \) determined by
\begin{align}
\II(X,Y) & = {\left\langle W(X), Y \right\rangle}_M \quad \text{for all} \ X, Y \in \Gamma(TM),
\end{align}
where \( {\left\langle {}\cdot{}, {}\cdot{} \right\rangle}_M \) denotes the induced Riemannian metric on \( M \). The \emph{extrinsic Gaussian curvature} and the \emph{mean curvature} of \( M \) are the smooth scalar fields on \( M \) defined by \( K^{\mathrm{e}} = \det W \) and \( H = \tr W \) respectively. The \emph{third fundamental form} of \( M \) is the symmetric smooth \( (0,2) \)-tensor field \( \III \) on \( M \) defined by \( \III(X,Y) = {\left\langle W(X), W(Y) \right\rangle} \), where \( X, Y \in \Gamma(TM) \).
\end{pdef}

\begin{pprop} \label{P:Weingarten-equation}
In the setup of Definition~\ref{D:SSFF}, we have
\begin{align}
W(X) & = - \tilde{\del}_X N
\end{align}
for any \( X \in \Gamma(TM) \) and any smooth extension of \( N \) to an open subset of \( \tilde{M} \).
\end{pprop}
\begin{proof}
Fix \( X, Y \in \Gamma(TM) \) and extend them arbitrarily smoothly to an open subset of \( \tilde{M} \). Since \( 2 {\big\langle \tilde{\del}_X N, N \big\rangle} = X {\left\langle N, N \right\rangle} = 0 \) along \( M \), we see that \( - \tilde{\del}_X N \in \Gamma(TM) \). Therefore, along \( M \), we have
\begin{align*}
\II(X,Y) & = {\big\langle N, \hat{\II}(X,Y) \big\rangle} = {\big\langle N, \tilde{\del}_X Y \big\rangle} = {\big\langle - \tilde{\del}_X N, Y \big\rangle} = {\big\langle - \tilde{\del}_X N, Y \big\rangle}_M,
\end{align*}
where \( {\left\langle {}\cdot{}, {}\cdot{} \right\rangle}_M \) denotes the induced Riemannian metric on \( M \).
\end{proof}

\begin{prmrk}
In Definition~\ref{D:Weingarten-map}, if we replace the orientation of \( M \) by the opposite one, then the extrinsic Gaussian curvature \( K^{\mathrm{e}} \) is multiplied by \( (-1)^{\dim M} \) and the mean curvature \( H \) changes its sign. If \( M \) is even-dimensional, then \( K^{\mathrm{e}} \) does not depend on the choice of an orientation of \( M \) (and the orientability of \( M \)).
\end{prmrk}

\subsection{Sectional curvature}

We further discuss the Gaussian curvature in the context of Riemann-Cartan geometry, which naturally introduces sectional curvature.

\begin{pprop}[Theorema Egregium, Riemann-Cartan version] \label{P:TE}
Let \( M \) be a smooth surface embedded in a flat Riemann-Cartan\/ \( 3 \)-manifold. Then \( K^{\mathrm{e}} = \frac{1}{2} \Scal \), which is an intrinsic quantity of \( M \).
\end{pprop}
\begin{proof}
Fix \( x \in M \) and let \( {\left\{ E_1, E_2 \right\}} \) be an orthonormal basis for \( T_x M \). By Proposition~\ref{P:SSFF}, we have \( R {\left( E_1, E_2, E_2, E_1 \right)} = \II {\left( E_1, E_1 \right)} \, \II {\left( E_2, E_2 \right)} - \II {\left( E_1, E_2 \right)} \, \II {\left( E_2, E_1 \right)} \). The left-hand side equals \( \frac{1}{2} \Scal \) and the right-hand side equals \( K^{\mathrm{e}} \).
\end{proof}

\noindent This motivates the following definition.

\begin{pdef}
The \emph{intrinsic Gaussian curvature} of a Riemann-Cartan surface is the half of its scalar curvature, i.e., \( K = \frac{1}{2} \Scal \).
\end{pdef}

\noindent This is closely related to sectional curvature as in Riemannian geometry.

\begin{pdef}
Let \( M \) be a Riemann-Cartan manifold. Let \( \Pi \) be a \( 2 \)-dimensional subspace of \( T_x M \) for some \( x \in M \). The \emph{\( \Pi \)-sectional curvature} of \( M \) is defined by
\begin{align}
\sec \! {\left( \Pi \right)} & = \frac{R {\left( u, v, v, u \right)}}{{\left\langle u, u \right\rangle} {\left\langle v, v \right\rangle} - {\left\langle u, v \right\rangle}{}^2},
\end{align}
where \( {\left\{ u, v \right\}} \) is any basis for \( \Pi \). This is independent of the choice of \( {\left\{ u, v \right\}} \).
\end{pdef}

\begin{pprop} \label{P:sectional-curvature}
Let \( M \) be a Riemann-Cartan surface. Then\/ \( \sec \! {\left( T_x M \right)} = K \) at every \( x \in M \).
\end{pprop}
\begin{proof}
Fix \( x \in M \) and let \( {\left\{ E_1, E_2 \right\}} \) be an orthonormal basis for \( T_x M \). Then we have \( \sec \! {\left( T_x M \right)} = R {\left( E_1, E_2, E_2, E_1 \right)} = \frac{1}{2} \Scal = K \).
\end{proof}

\begin{pprop} \label{P:Gaussian-curvature}
Let \( M \) be a smooth surface embedded in a Riemann-Cartan\/ \( 3 \)-manifold \( \tilde{M} \). Then\/ \( \widetilde{\sec} {\left( T_x M \right)} = K - K^{\mathrm{e}} \) at every \( x \in M \).
\end{pprop}
\begin{proof}
Fix \( x \in M \) and let \( {\left\{ E_1, E_2 \right\}} \) be an orthonormal basis for \( T_x M \). By Proposition~\ref{P:SSFF}, we have \( \tilde{R} {\left( E_1, E_2, E_2, E_1 \right)} = R {\left( E_1, E_2, E_2, E_1 \right)} - K^{\mathrm{e}} \).
\end{proof}

The Gauss-Bonnet theorem works in Riemann-Cartan geometry as well.

\begin{plem} \label{L:Euler-class}
Let \( M \) be an oriented Riemann-Cartan surface. Let \( FM \surjto M \) be the oriented orthonormal frame bundle, and let \( \omega \in \Omega^1 {\left( FM, \mathfrak{so}(2) \right)} \) be the connection. Then \( K \, \mathrm{d} a = s^* \mathrm{d} \omega^1_2 \) for any local oriented orthonormal smooth frame \( s \) for \( M \), where\/ \( \mathrm{d} a \) is the (Riemannian) area form on \( M \).
\end{plem}
\begin{proof}
Let \( s = {\left( E_1, E_2 \right)} \) and \( \upsilon = s^* \omega^1_2 \). By Proposition~\ref{P:sectional-curvature}, we have
\begin{align*}
K & = R {\left( E_1, E_2, E_2, E_1 \right)} = {\left\langle \del_{E_1} \del_{E_2} E_2 - \del_{E_2} \del_{E_1} E_2 - \del_{{\left[ E_1, E_2 \right]}} E_2, E_1 \right\rangle} \\
& = {\left\langle \del_{E_1} {\left( \upsilon {\left( E_2 \right)} E_1 \right)} - \del_{E_2} {\left( \upsilon {\left( E_1 \right)} E_1 \right)} - \upsilon {\left( {\left[ E_1, E_2 \right]} \right)} E_1, E_1 \right\rangle} \\
& = {\left\langle \begin{aligned}
& E_1 {\left( \upsilon {\left( E_2 \right)} \right)} E_1 + \upsilon {\left( E_2 \right)} \del_{E_1} E_1 \\
& - E_2 {\left( \upsilon {\left( E_1 \right)} \right)} E_1 - \upsilon {\left( E_1 \right)} \del_{E_2} E_1 - \upsilon {\left( {\left[ E_1, E_2 \right]} \right)} E_1,
\end{aligned} E_1 \right\rangle} \\
& = E_1 {\left( \upsilon {\left( E_2 \right)} \right)} - E_2 {\left( \upsilon {\left( E_1 \right)} \right)} - \upsilon {\left( {\left[ E_1, E_2 \right]} \right)} = \mathrm{d} \upsilon {\left( E_1, E_2 \right)}.
\end{align*}
This completes the proof.
\end{proof}

\begin{plem} \label{L:gauge-invariance}
Let \( M \) be an oriented Riemannian surface. Let \( FM \surjto M \) be the oriented orthonormal frame bundle. If \( \omega, \omega' \in \Omega^1 {\left( FM, \mathfrak{so}(2) \right)} \) are two connections, then \( s^* {\left( \omega - \omega' \right)} \) does not depend on the choice of a local oriented orthonormal smooth frame \( s \) for \( M \), so it produces a well-defined form\/ \( {\left( \omega - \omega' \right)}^{\mathsf{p}} \in \Omega^1 {\left( M, \mathfrak{so}(2) \right)} \).
\end{plem}
\begin{proof}
For any gauge transformation \( \varphi \colon FM \isoto FM \) with the associated map \( g_{\varphi} \colon FM \to \mathrm{SO}(2) \), we have \( \varphi^* {\left( \omega - \omega' \right)} = \Ad_{g_{\varphi}^{-1}} {\left( \omega - \omega' \right)} = \omega - \omega' \) by Lemma~\ref{L:gauge-transformation-connection} and the elementary fact that \( BAB^{-1} = A \) for all \( A \in \mathfrak{so}(2) \) and \( B \in \mathrm{SO}(2) \).
\end{proof}

\begin{pprop}[The Gauss-Bonnet theorem, Riemann-Cartan version] \label{P:Gauss-Bonnet}
Let \( M \) be a closed oriented Riemann-Cartan surface. Then
\begin{align}
\int_M K \, \mathrm{d} a & = 2 \boldsymbol{\pi} \chi(M),
\end{align}
where\/ \( \mathrm{d} a \) is the (Riemannian) area form on \( M \).
\end{pprop}
\begin{proof}
Note that \( K \, \mathrm{d} a = \bar{K} \, \mathrm{d} a + \mathrm{d} {\left( {\left( \omega - \bar{\omega} \right)}^{\mathsf{p}} \right)}^1_2 \) by Lemmas~\ref{L:Euler-class} and~\ref{L:gauge-invariance}. Then the equality follows from the original Gauss-Bonnet theorem and Stokes' theorem.
\end{proof}

\noindent The Chern-Gauss-Bonnet theorem originated in~\cite{MR0014760} can be generalized to any Riemannian vector bundles with any metric-compatible connection (Theorem~1 of~\cite{MR2261968}). Proposition~\ref{P:Gauss-Bonnet} agrees with a particular case of this generalization.

\subsection{The Weitzenb{\"o}ck case}

Hypersurfaces in a Weitzenb{\"o}ck manifold have a powerful tool: the \emph{Gauss map}.

\begin{pdef}[Gauss map] \label{D:Gauss-map}
Let \( M \) be an oriented smooth hypersurface embedded in an oriented Riemannian \( n \)-manifold \( \tilde{M} \), where \( M \) is oriented by the unit normal vector field \( N \) along \( M \). Let \( s = {\left( E_1, \dotsc, E_n \right)} \) be a local oriented orthonormal smooth frame for \( \tilde{M} \) on \( U \). The \emph{Gauss map} of \( M \) with respect to \( s \) is the smooth map \( n^s = {\left( N^1, \dotsc, N^n \right)} \colon U \cap M \to \mathbb{S}^{n-1} \), where \( N = N^1 E_1 + \dotsb + N^n E_n \).
\end{pdef}

\noindent In particular, if \( M \) is an oriented smooth hypersurface embedded in a Weitzenb{\"o}ck \( n \)-manifold \( \tilde{M} \), then there is a globally well-defined Gauss map \( n \colon M \to \mathbb{S}^{n-1} \). In fact, the Weitzenb{\"o}ck Weingarten map is given by the differential of the Gauss map in the following sense.

\begin{pprop} \label{P:Weingarten-map}
Let \( M \) be an oriented smooth hypersurface embedded in a Weitzenb{\"o}ck \( n \)-manifold \( \tilde{M} \) with global smooth frame \( s = {\left( E_1, \dotsc, E_n \right)} \), where \( M \) is oriented by the unit normal vector field \( N = N^1 E_1 + \dotsb + N^n E_n \) along \( M \). Then
\begin{align}
W^s & = - \sum_{i = 1}^n \mathrm{d} N^i \otimes E_i.
\end{align}
\end{pprop}
\begin{proof}
This is an immediate corollary of Proposition~\ref{P:Weingarten-equation}.
\end{proof}

\begin{pcor}
In the setup of Proposition~\ref{P:Weingarten-map}, \( M \) is totally geodesic in \( \tilde{M} \) iff the Gauss map \( n^s \colon M \to \mathbb{S}^{n-1} \) is constant on each component of \( M \).
\end{pcor}
\begin{proof}
Note that \( \II^s = 0 \) iff \( W^s = 0 \) iff \( \mathrm{d} N^i = 0 \) for all \( i \in \{ 1, \dotsc, n \} \).
\end{proof}

\begin{pcor}
In the setup of Proposition~\ref{P:Weingarten-map}, suppose that \( s' \) is another global smooth frame for \( \tilde{M} \) such that \( s = s' \) along \( M \). Then
\begin{align}
\begin{aligned}
W^s & = W^{s'}, & \quad \II^s & = \II^{s'}, & \quad \III^s & = \III^{s'}, \\
\tau^s & = \tau^{s'}, & \quad {\left( K^{\mathrm{e}} \right)}{}^s & = {\left( K^{\mathrm{e}} \right)}{}^{s'}, & \quad H^s & = H^{s'},
\end{aligned}
\end{align}
where the superscripts denote the corresponding Weitzenb{\"o}ck structure.
\end{pcor}
\begin{proof}
We know \( W^s = W^{s'} \). All the others follow from this.
\end{proof}

The Weitzenb{\"o}ck mean curvature and torsion \( 2 \)-form are given as follows.

\begin{pprop} \label{P:H}
In the setup of Proposition~\ref{P:Weingarten-map}, let \( \iota \colon M \injto \tilde{M} \) be the inclusion and let\/ \( {\left( \varepsilon^1, \dotsc, \varepsilon^n \right)} \) be the dual of \( s \). Then
\begin{align}
H^s & = - \sum_{i = 1}^n E_i^{\top} {\left( N^i \right)} \quad \text{and} \quad \tau^s = \sum_{i = 1}^n N^i \iota^* \mathrm{d} \varepsilon^i.
\end{align}
\end{pprop}
\begin{proof}
If \( {\left( \bar{E}_1, \dotsc, \bar{E}_{n-1} \right)} \) is a local oriented orthonormal smooth frame for \( M \),
\begin{align*}
H^s & = \tr W^s = \sum_{j = 1}^{n-1} {\left\langle \bar{E}_j, - \sum_{i = 1}^n \bar{E}_j {\left( N^i \right)} E_i \right\rangle} = - \sum_{i = 1}^n \sum_{j = 1}^{n-1} {\left\langle \bar{E}_j, E_i \right\rangle} \bar{E}_j {\left( N^i \right)}.
\end{align*}
Here, \( E_i^{\top} = \sum_{j = 1}^{n-1} {\left\langle \bar{E}_j, E_i \right\rangle} \bar{E}_j \). The former follows. For all \( i, j \in \{ 1, \dotsc, n \} \),
\begin{align*}
\sum_{k = 1}^n N^k \, \mathrm{d} \varepsilon^k {\left( E_i, E_j \right)} & = - \sum_{k = 1}^n N^k \varepsilon^k {\left( {\left[ E_i, E_j \right]} \right)} = - {\left\langle N, {\left[ E_i, E_j \right]} \right\rangle} = {\big\langle N, \tilde{T}^s {\left( E_i, E_j \right)} \big\rangle}.
\end{align*}
The latter follows.
\end{proof}

The Gaussian curvature for the Weitzenb{\"o}ck case is given as follows.

\begin{pprop}
Let \( M \) be an oriented smooth surface embedded in a Weitzen\-b{\"o}ck\/ \( 3 \)-manifold \( \tilde{M} \) with global smooth frame \( s = {\left( E_1, E_2, E_3 \right)} \), where \( M \) is oriented by the unit normal vector field \( N = N^1 E_1 + N^2 E_2 + N^3 E_3 \) along \( M \). Then
\begin{align}
K \, \mathrm{d} a & = K^{\mathrm{e}} \, \mathrm{d} a = \sum_{\mathrm{cyc}} N^i \, \mathrm{d} N^j \wedge \mathrm{d} N^k = {\left( n^s \right)}^* \, \mathrm{d} a_{\mathbb{S}^2},
\end{align}
where\/ \( \mathrm{d} a \) is the (Riemannian) area form on \( M \) and\/ \( \mathrm{d} a_{\mathbb{S}^2} \) is the standard area form on the unit sphere\/ \( \mathbb{S}^2 \).
\end{pprop}
\begin{proof}
The first equality follows from Proposition~\ref{P:TE}. For the second equality, let \( {\left( \bar{E}_1, \bar{E}_2 \right)} \) be a local oriented orthonormal smooth frame for \( M \). Let \( g \) be the \( \mathrm{SO}(3) \)-valued smooth map given by \( {\left( E_1, E_2, E_3 \right)} = {\left( \bar{E}_1, \bar{E}_2, N \right)} \cdot g \). By Proposition~\ref{P:Weingarten-map},
\begin{align*}
{\left( W^s \right)}^i_j & = {\left\langle \bar{E}_i, - \sum_{k = 1}^3 \mathrm{d} N^k {\left( \bar{E}_j \right)} E_k \right\rangle} = - \sum_{k = 1}^3 g^i_k \, \mathrm{d} N^k {\left( \bar{E}_j \right)}
\end{align*}
relative to \( {\left( \bar{E}_1, \bar{E}_2 \right)} \) for all \( i, j \in \{ 1, 2 \} \), so that
\begin{align*}
K^{\mathrm{e}} & = \det W^s = \sum_{i, j = 1}^3 g^1_i \, \mathrm{d} N^i \wedge g^2_j \, \mathrm{d} N^j {\left( \bar{E}_1, \bar{E}_2 \right)} \\
& = \sum_{\substack{i, j, k \\ \mathrm{cyc}}} {\left( g^1_i g^2_j - g^1_j g^2_i \right)} \, \mathrm{d} N^i \wedge \mathrm{d} N^j {\left( \bar{E}_1, \bar{E}_2 \right)} = \sum_{\mathrm{cyc}} g^3_k \, \mathrm{d} N^i \wedge \mathrm{d} N^j {\left( \bar{E}_1, \bar{E}_2 \right)}.
\end{align*}
Here, \( g^3_k = N^k \). For the last equality, observe that
\begin{align*}
\sum_{\mathrm{cyc}} N^i \, \mathrm{d} N^j \wedge \mathrm{d} N^k & = {\left( n^s \right)}^* \iota_{\mathbb{S}^2}^* {\left( x \, \mathrm{d} y \wedge \mathrm{d} z + y \, \mathrm{d} z \wedge \mathrm{d} x + z \, \mathrm{d} x \wedge \mathrm{d} y \right)} = {\left( n^s \right)}^* \, \mathrm{d} a_{\mathbb{S}^2},
\end{align*}
where \( \iota_{\mathbb{S}^2} \colon \mathbb{S}^2 \injto \mathbb{R}^3 \) is the inclusion.
\end{proof}

\noindent The Gauss-Bonnet theorem (Proposition~\ref{P:Gauss-Bonnet}) yields the following topological characterization of the Gauss map.

\begin{pcor}
Let \( M \) be a closed oriented smooth surface embedded in a Weitzen\-b{\"o}ck\/ \( 3 \)-manifold. Let \( n \colon M \to \mathbb{S}^2 \) be the Gauss map. Then\/ \( 2 \deg n = \chi(M) \).
\end{pcor}
\begin{proof}
We have
\begin{align}
2 \boldsymbol{\pi} \chi(M) & = \int_M K \, \mathrm{d} a = \int_M n^* \, \mathrm{d} a_{\mathbb{S}^2} = {\left( \deg n \right)} \int_{\mathbb{S}^2} \mathrm{d} a_{\mathbb{S}^2} = 4 \boldsymbol{\pi} \deg n
\end{align}
by Proposition~\ref{P:Gauss-Bonnet}.
\end{proof}

\noindent In the Euclidean case (i.e., if \( M \) is embedded in \( \mathbb{R}^3 \)), this is a well-known classical result. In fact, historically, the degree of the Gauss map was a key in proving the classical Gauss-Bonnet theorem (see Section~3 in Chapter~5 of~\cite{MR0448362}).

\section{Surfaces in a Riemann-Cartan 3-manifold} \label{S:surfaces}

In this section, we focus on surfaces in a Riemann-Cartan \( 3 \)-manifold. Throughout this section, we look at an oriented smooth surface \( S \) embedded in an oriented Riemann-Cartan \( 3 \)-manifold \( M \). The main interest of this section is the mean curvature and torsion \( 2 \)-form in Riemann-Cartan geometry. We will encounter a number of results addressing the following complex-valued quantity on surfaces.

\begin{pdef} \label{D:H}
We define a complex-valued smooth scalar field \( \boldsymbol{H} \) on \( S \) by
\begin{align}
\boldsymbol{H} & = H + \boldsymbol{i} {\star} \tau,
\end{align}
where \( H \) is the mean curvature of \( S \) (Definition~\ref{D:Weingarten-map}) and \( {\star} \tau \) is the \( S \)-Hodge dual of the torsion \( 2 \)-form of \( S \) (Definition~\ref{D:torsion-form}) in Riemann-Cartan geometry.
\end{pdef}

\subsection{The hat map}

For oriented Riemannian \( 3 \)-manifolds, there is a useful tool called the \emph{hat map}, which often appears in physics and engineering literature, e.g.,~\cite{MR1300410}. It provides a natural identification between \( \mathbb{R}^3 \) and \( \mathfrak{so}(3) \).

\begin{pdef}
The \emph{cross product matrix} of \( a = {\left( a^1, a^2, a^3 \right)} \in \mathbb{R}^3 \) is the matrix \( \hat{a} \in \mathfrak{so}(3) \) corresponding to the linear map \( \mathbb{R}^3 \to \mathbb{R}^3 \) defined by \( x \mapsto a \times x \), i.e.,
\begin{align}
\hat{a} & = \begin{pmatrix*}[c]
0 & -a^3 & a^2 \\
a^3 & 0 & -a^1 \\
-a^2 & a^1 & 0
\end{pmatrix*}.
\end{align}
The \emph{hat map} is the real vector space isomorphism \( \mathbb{R}^3 \to \mathfrak{so}(3) \) given by \( a \mapsto \hat{a} \).
\end{pdef}

\noindent The standard basis for \( \mathbb{R}^3 \) is mapped to the following basis for \( \mathfrak{so}(3) \).
\begin{align}
L_1 & = \begin{pmatrix*}[r]
0 & 0 & 0 \\
0 & 0 & -1 \\
0 & 1 & 0
\end{pmatrix*}, \qquad L_2 = \begin{pmatrix*}[r]
0 & 0 & 1 \\
0 & 0 & 0 \\
-1 & 0 & 0
\end{pmatrix*}, \qquad L_3 = \begin{pmatrix*}[r]
0 & -1 & 0 \\
1 & 0 & 0 \\
0 & 0 & 0
\end{pmatrix*}. \label{E:L}
\end{align}
The standard inner product on \( \mathbb{R}^3 \) corresponds to \( - \frac{1}{2} \tr \), i.e.,
\begin{align}
a \cdot b & = - \frac{1}{2} \tr \! {\big( \hat{a} \hat{b} \big)} \quad \text{for all} \ a, b \in \mathbb{R}^3.
\end{align}

Some elementary properties of the hat map are listed in the following.

\begin{plem} \label{L:hat}
Let \( a, b \in \mathbb{R}^3 \).
\begin{enumerate}
\item \( \hat{a} b = a \times b = - b \times a = - \hat{b} a \) and \( \widehat{a \times b} = \hat{a} \hat{b} - \hat{b} \hat{a} \).
\item \( \hat{a} \hat{b} \hat{a} = - {\left( a \cdot b \right)} \hat{a} \). If \( a \) is unit, then \( \hat{a}^3 = - \hat{a} \).
\item \( \widehat{Ab} = A \hat{b} A^{-1} \) for any \( A \in \mathrm{SO}(3) \).
\item \( \widehat{Ab} = A \hat{b} - \hat{b} A \) for any \( A \in \mathfrak{so}(3) \).
\end{enumerate}
\end{plem}
\begin{proof}
(a) is elementary. (b) follows from the following observation.
\begin{align*}
a \times {\left( b \times {\left( a \times x \right)} \right)} & = a \times {\left( {\left( x \cdot b \right)} a - {\left( a \cdot b \right)} x \right)} = - {\left( a \cdot b \right)} a \times x \quad \text{for any} \ x \in \mathbb{R}^3.
\end{align*}
For~(c), let \( A \in \mathrm{SO}(3) \). Since \( A \hat{b} A^{-1} = \Ad_A \hat{b} \in \mathfrak{so}(3) \), it suffices to compare the entries \( (i,j) \in \{ (1,2), (2,3), (3,1) \} \). For these \( (i,j) \),
\begin{align*}
{\big( A \hat{b} A^{-1} \big)}^i_j & = \sum_{\ell, m = 1}^3 A^i_{\ell} \hat{b}^{\ell}_m A^j_m = \sum_{\substack{\ell, m, n \\ \mathrm{cyc}}} {\left( A^i_{\ell} A^j_m - A^i_m A^j_{\ell} \right)} \hat{b}^{\ell}_m = \sum_{n = 1}^3 - A^k_n b^n = - (Ab)^k,
\end{align*}
where \( k \in \{ 1, 2, 3 \} \setminus \{ i, j \} \). (c) follows. (d) follows from~(a).
\end{proof}

One of the major applications of the hat map is \emph{Rodrigues' rotation formula}. Every matrix in \( \mathrm{SO}(3) \) has an axis-angle representation. The spatial rotation of the angle \( \theta \in \mathbb{R} / 2 \boldsymbol{\pi} \mathbb{Z} \) about the axis \( e \in \mathbb{S}^2 \) is given by the following matrix.
\begin{align}
\boldsymbol{e}^{\theta \hat{e}} & = I + \sin(\theta) \hat{e} + {\left( 1 - \cos(\theta) \right)} \hat{e}^2 \in \mathrm{SO}(3),
\end{align}
where \( I \) is the identity matrix. We state one lemma for later use.

\begin{plem} \label{L:rotation}
Let \( U \) be an open set in\/ \( \mathbb{R}^n \). Let \( \theta \colon U \to \mathbb{R} / 2 \boldsymbol{\pi} \mathbb{Z} \) and \( e \colon U \to \mathbb{S}^2 \) be smooth maps. Then the pullback of the Maurer-Cartan\/ \( 1 \)-form \( \mu \) on\/ \( \mathrm{SO}(3) \) via the smooth map \( g = \boldsymbol{e}^{\theta \hat{e}} \colon U \to \mathrm{SO}(3) \) is given by
\begin{align}
g^* \mu & = \mathrm{d} \theta \hat{e} + \sin(\theta) \, \mathrm{d} \hat{e} + {\left( 1 - \cos(\theta) \right)} \, \widehat{\mathrm{d} e \times e}.
\end{align}
\end{plem}
\begin{proof}
Fix a vector \( X \) in \( U \), and let the prime in the following computation denote the \( X \)-derivative. By Lemma~\ref{L:hat}, we have \( \hat{e}^3 = - \hat{e} \) and \( \hat{e} \hat{e}' \hat{e} = 0 \). Observe that
\begin{align*}
g^{-1} g' & = \begin{aligned}[t]
& {\left( I - \sin(\theta) \hat{e} + {\left( 1 - \cos(\theta) \right)} \hat{e}^2 \right)} \\
& \cdot {\left( \cos(\theta) \theta' \hat{e} + \sin(\theta) \hat{e}' + \sin(\theta) \theta' \hat{e}^2 + {\left( 1 - \cos(\theta) \right)} {\left( \hat{e}' \hat{e} + \hat{e} \hat{e}' \right)} \right)}
\end{aligned} \\
& = \begin{aligned}[t]
& {\left( \begin{aligned}
& \cos(\theta) I - \sin(\theta) \cos(\theta) \hat{e} - \cos(\theta) {\left( 1 - \cos(\theta) \right)} I \\
& + \sin(\theta) \hat{e} + \sin^2(\theta) I - \sin(\theta) {\left( 1 - \cos(\theta) \right)} \hat{e}
\end{aligned} \right)} \theta' \hat{e} \\
& + {\left( \begin{aligned}
& \sin(\theta) I - \sin^2(\theta) \hat{e} + \sin(\theta) {\left( 1 - \cos(\theta) \right)} \hat{e}^2 \\
& + {\left( 1 - \cos(\theta) \right)} \hat{e} - \sin(\theta) {\left( 1 - \cos(\theta) \right)} \hat{e}^2 - {\left( 1 - \cos(\theta) \right)}^2 \hat{e}
\end{aligned} \right)} \hat{e}' \\
& + {\left( 1 - \cos(\theta) \right)} \hat{e}' \hat{e}
\end{aligned} \\
& = \theta' \hat{e} + {\left( \sin(\theta) I - {\left( 1 - \cos(\theta) \right)} \hat{e} \right)} \hat{e}' + {\left( 1 - \cos(\theta) \right)} \hat{e}' \hat{e}.
\end{align*}
By Lemma~\ref{L:hat}, we have \( \widehat{e' \times e} = \hat{e}' \hat{e} - \hat{e} \hat{e}' \). This completes the proof.
\end{proof}

Note that \( \mathbb{R}^3 \)-valued forms can be viewed as \( \mathfrak{so}(3) \)-valued forms via the identification of the hat map. We have the following useful lemma.

\begin{plem}
Let \( P \) be a smooth manifold. Let \( \alpha \in \Omega^p {\left( P, \mathfrak{so}(3) \right)} \) and \( \beta \in \Omega^q {\left( P, \mathbb{R}^3 \right)} \). Then \( \widehat{\alpha \wedge \beta} = {\big[ \alpha \wedge \hat{\beta} \big]} \in \Omega^{p+q} {\left( P, \mathfrak{so}(3) \right)} \).
\end{plem}
\begin{proof}
For each \( (i,j,k) \in \{ (1,2,3), (2,3,1), (3,1,2) \} \), we have
\begin{align*}
{\big[ \alpha \wedge \hat{\beta} \big]}^i_j & = {\big( \alpha \wedge \hat{\beta} \big)}^i_j - (-1)^{pq} {\big( \hat{\beta} \wedge \alpha \big)}^i_j = \alpha^i_k \wedge \hat{\beta}^k_j - (-1)^{pq} \hat{\beta}^i_k \wedge \alpha^k_j \\
& = \alpha^i_k \wedge \hat{\beta}^k_j - \alpha^k_j \wedge \hat{\beta}^i_k = - \alpha^k_i \wedge \beta^i - \alpha^k_j \wedge \beta^j - \alpha^k_k \wedge \beta^k = - {\left( \alpha \wedge \beta \right)}^k.
\end{align*}
This completes the proof.
\end{proof}

\noindent Then~\eqref{E:fundamental-equations} corresponds to the following.
\begin{align}
\begin{aligned}
\mathrm{d} \hat{\theta} + {\big[ \omega \wedge \hat{\theta} \big]} & = \hat{\Theta}, & \quad \mathrm{d} \hat{\Theta} + {\big[ \omega \wedge \hat{\Theta} \big]} & = {\big[ \Omega \wedge \hat{\theta} \big]}, \\
\mathrm{d} \omega + {\textstyle \frac{1}{2}} {\left[ \omega \wedge \omega \right]} & = \Omega, & \quad \mathrm{d} \Omega + {\left[ \omega \wedge \Omega \right]} & = 0.
\end{aligned}
\end{align}
If \( \varphi \colon FM \isoto FM \) is a gauge transformation of the oriented orthonormal frame bundle \( FM \surjto M \) over an oriented Riemann-Cartan \( 3 \)-manifold \( M \), then the four fundamental forms obey the following gauge transformation formulas.
\begin{align}
\begin{aligned}
\varphi^* \hat{\theta} & = \Ad_{g_{\varphi}^{-1}} \hat{\theta}, & \quad \varphi^* \hat{\Theta} & = \Ad_{g_{\varphi}^{-1}} \hat{\Theta}, \\
\varphi^* \omega & = \Ad_{g_{\varphi}^{-1}} \omega + g_{\varphi}^* \mu, & \quad \varphi^* \Omega & = \Ad_{g_{\varphi}^{-1}} \Omega,
\end{aligned}
\end{align}
where \( g_{\varphi} \colon FM \to \mathrm{SO}(3) \) is the associated map and \( \mu \) is the Maurer-Cartan \( 1 \)-form on \( \mathrm{SO}(3) \). Only the connection \( \omega \) involves an additional term.

\subsection{Divergence and curl} \label{SS:divergence-curl}

In this subsection, we are interested in the case of an oriented smooth surface \( S \) embedded in a Weitzenb{\"o}ck \( 3 \)-manifold \( M \) with global smooth frame \( s = {\left( E_1, E_2, E_3 \right)} \), where \( S \) is oriented by the unit normal vector field \( N = N^1 E_1 + N^2 E_2 + N^3 E_3 \) along \( S \). We will find a description of the complex-valued quantity \( \boldsymbol{H} \) by the Gauss map and a nice gauge transformation formula.

To avoid an abuse of notation, we first introduce the following notations.

\begin{pdef}
Let \( \Sigma \) be a smooth manifold. Let \( d = {\left( D_1, D_2, D_3 \right)} \in \Gamma {\left( T \Sigma \right)}{}^3 \) be a triple of smooth vector fields. Let \( f \colon \Sigma \to \mathbb{R} \) and \( F = {\left( F^1, F^2, F^3 \right)} \colon \Sigma \to \mathbb{R}^3 \) be smooth maps. We define three smooth maps \( \Grad_d (f) \colon \Sigma \to \mathbb{R}^3 \), \( \Div_d (F) \colon \Sigma \to \mathbb{R} \), and \( \Curl_d (F) \colon \Sigma \to \mathbb{R}^3 \) by
\begin{align}
\begin{gathered}
\Grad_d (f) = {\left( D_1 f, D_2 f, D_3 f \right)}, \qquad \Div_d (F) = D_1 F^1 + D_2 F^2 + D_3 F^3, \\
\Curl_d (F) = {\left( D_2 F^3 - D_3 F^2, D_3 F^1 - D_1 F^3, D_1 F^2 - D_2 F^1 \right)}.
\end{gathered}
\end{align}
A simple mnemonic is the usual gradient, divergence, and curl in calculus.
\end{pdef}

We get a hint from Proposition~\ref{P:H} that the Weitzenb{\"o}ck mean curvature \( H^s \) can be understood as a ``divergence'' of the Gauss map \( n^s \) with respect to the projection of the global smooth frame \( s \). This agrees with the classical fact that the mean curvature in Euclidean geometry equals the minus of the divergence of the Gauss map. We shall observe that the Weitzenb{\"o}ck torsion \( 2 \)-form \( \tau^s \) can be understood as a ``curl'' of the Gauss map \( n^s \).

\begin{plem}
We have
\begin{align}
\tau^s & = - \sum_{\mathrm{cyc}} N^i {\left\langle N, {\left[ E_j, E_k \right]} \right\rangle} \, \mathrm{d} a.
\end{align}
\end{plem}
\begin{proof}
Let \( {\left( \bar{E}_1, \bar{E}_2 \right)} \) be a local oriented orthonormal smooth frame for \( S \). Let \( \bar{g} \) be the \( \mathrm{SO}(3) \)-valued smooth map given by \( {\left( \bar{E}_1, \bar{E}_2, N \right)} = {\left( E_1, E_2, E_3 \right)} \cdot \bar{g} \). Then
\begin{align*}
\tau^s {\left( \bar{E}_1, \bar{E}_2 \right)} & = \sum_{i, j = 1}^3 \bar{g}^i_1 \bar{g}^j_2 {\left\langle N, T^s {\left( E_i, E_j \right)} \right\rangle} = \sum_{\substack{i, j, k \\ \mathrm{cyc}}} {\big( \bar{g}^i_1 \bar{g}^j_2 - \bar{g}^i_2 \bar{g}^j_1 \big)} {\left\langle N, T^s {\left( E_i, E_j \right)} \right\rangle}.
\end{align*}
Here, \( \bar{g}^i_1 \bar{g}^j_2 - \bar{g}^i_2 \bar{g}^j_1 = \bar{g}^k_3 = N^k \) and \( T^s {\left( E_i, E_j \right)} = - {\left[ E_i, E_j \right]} \).
\end{proof}

\begin{pprop} \label{P:divergence-curl}
We have
\begin{align}
H^s & = - \Div_{s^{\top}} \! {\left( n^s \right)} \quad \text{and} \quad \tau^s = - \Curl_{s^{\top}} \! {\left( n^s \right)} \cdot n^s \, \mathrm{d} a.
\end{align}
\end{pprop}
\begin{proof}
The former is merely a rephrasing of Proposition~\ref{P:H}. Look at the latter. For an arbitrary smooth extension of \( N \) to an open subset of \( M \), we have
\begin{align*}
\sum_{\mathrm{cyc}} N^i {\left\langle N, {\left[ E_j, E_k \right]} \right\rangle} & = \sum_{\mathrm{cyc}} N^i {\left\langle N, {\left[ E_j^{\top} + N^j N, E_k^{\top} + N^k N \right]} \right\rangle} \\
& = \sum_{\mathrm{cyc}} N^i {\left( E_j^{\top} {\left( N^k \right)} - E_k^{\top} {\left( N^j \right)} + N^j F_k - N^k F_j \right)},
\end{align*}
where \( F_k = N {\left( N^k \right)} - {\left\langle N, {\left[ E_k^{\top}, N \right]} \right\rangle} \). Note that the summation of the last two terms containing \( F_{\bullet} \) vanishes, since \( \sum_{\mathrm{cyc}} N^i N^k F_j = \sum_{\mathrm{cyc}} N^j N^i F_k \).
\end{proof}

Note that \( M \) has the \emph{cross product} defined by \( X \times Y = {\left( {\star} {\left( X^{\flat} \wedge Y^{\flat} \right)} \right)}{}^{\sharp} \) for any vectors \( X \) and \( Y \) (at the same point). For any vector \( X \in T_x M \) at \( x \in S \), we define
\begin{align}
X^{\times} & = N \times X \in T_x S.
\end{align}
It is worth mentioning that the restriction of the map \( X \mapsto X^{\times} \) to \( TS \) actually coincides with the almost complex structure on \( S \) induced by its conformal structure, which is an intrinsic structure on \( S \).

\begin{plem}
We have
\begin{align}
\Curl_{s^{\top}} \! {\left( n^s \right)} & = - \Div_{s^{\times}} \! {\left( n^s \right)} \, n^s \quad \text{and} \quad \Curl_{s^{\times}} \! {\left( n^s \right)} = \Div_{s^{\top}} \! {\left( n^s \right)} \, n^s.
\end{align}
\end{plem}
\begin{proof}
For each \( (i,j,k) \in \{ (1,2,3), (2,3,1), (3,1,2) \} \), we have
\begin{align*}
& E_i^{\times} = N \times {\left( E_j \times E_k \right)} = N^k E_j - N^j E_k = N^k E_j^{\top} - N^j E_k^{\top}, \\
& - \Div_{s^{\times}} \! {\left( n^s \right)} \, N^i = - N^i E_i^{\times} {\left( N^i \right)} - N^i E_j^{\times} {\left( N^j \right)} - N^i E_k^{\times} {\left( N^k \right)} \\
& \qquad = {\left( N^j E_i^{\times} {\left( N^j \right)} + N^k E_i^{\times} {\left( N^k \right)} \right)} - N^i E_j^{\times} {\left( N^j \right)} - N^i E_k^{\times} {\left( N^k \right)} \\
& \qquad = {\left( N^j E_i^{\times} - N^i E_j^{\times} \right)} {\left( N^j \right)} - {\left( N^i E_k^{\times} - N^k E_i^{\times} \right)} {\left( N^k \right)} \\
& \qquad = E_k^{\times \! \times} {\left( N^j \right)} - E_j^{\times \! \times} {\left( N^k \right)} = - E_k^{\top} {\left( N^j \right)} + E_j^{\top} {\left( N^k \right)} = {\left( \Curl_{s^{\top}} \! {\left( n^s \right)} \right)}^i, \\
& {\left( \Curl_{s^{\times}} \! {\left( n^s \right)} \right)}^i = E_j^{\times} {\left( N^k \right)} - E_k^{\times} {\left( N^j \right)} \\
& \qquad = {\left( N^i E_k^{\top} - N^k E_i^{\top} \right)} {\left( N^k \right)} - {\left( N^j E_i^{\top} - N^i E_j^{\top} \right)} {\left( N^j \right)} \\
& \qquad = - {\left( N^j E_i^{\top} {\left( N^j \right)} + N^k E_i^{\top} {\left( N^k \right)} \right)} + N^i E_j^{\top} {\left( N^j \right)} + N^i E_k^{\top} {\left( N^k \right)} \\
& \qquad = N^i E_i^{\top} {\left( N^i \right)} + N^i E_j^{\top} {\left( N^j \right)} + N^i E_k^{\top} {\left( N^k \right)} = \Div_{s^{\top}} \! {\left( n^s \right)} \, N^i.
\end{align*}
The desired equalities follow.
\end{proof}

\begin{pcor} \label{C:divergence-curl}
The following four hold.
\begin{align}
\begin{aligned}
\Div_{s^{\top}} \! {\left( n^s \right)} & = - H^s, & \quad \Curl_{s^{\top}} \! {\left( n^s \right)} & = - {\star} \tau^s n^s, \\
\Div_{s^{\times}} \! {\left( n^s \right)} & = {\star} \tau^s, & \quad \Curl_{s^{\times}} \! {\left( n^s \right)} & = - H^s n^s.
\end{aligned}
\end{align}
\end{pcor}

Now, we want to find the gauge transformation formula.

\begin{plem} \label{L:gauge-transformation}
Let \( g \colon M \to \mathrm{SO}(3) \) be a smooth map. Then
\begin{align}
H^{s \cdot g} & = H^s - \sum_{\mathrm{cyc}} {\left( g^{-1} \right)}^* \mu^i_j {\left( E_k^{\times} \right)} \quad \text{and} \quad {\star} \tau^{s \cdot g} = {\star} \tau^s - \sum_{\mathrm{cyc}} {\left( g^{-1} \right)}^* \mu^i_j {\left( E_k^{\top} \right)},
\end{align}
where \( \mu \) is the Maurer-Cartan\/ \( 1 \)-form on\/ \( \mathrm{SO}(3) \).
\end{plem}
\begin{proof}
By Corollary~\ref{C:divergence-curl}, we have
\begin{align*}
H^{s \cdot g} & = - \sum_{i, j, k = 1}^3 g^j_i E_j^{\top} {\left\langle N, g^k_i E_k \right\rangle} = - \sum_{i, j, k = 1}^3 {\left( g^j_i g^k_i E_j^{\top} {\left( N^k \right)} + g^j_i E_j^{\top} \! {\left( g^k_i \right)} N^k \right)} \\
& = - \sum_{i = 1}^3 E_i^{\top} {\left( N^i \right)} - \sum_{i, j, k = 1}^3 g^j_i \, \mathrm{d} g^k_i {\left( N^k E_j^{\top} \right)} = H^s - \sum_{j, k = 1}^3 {\left( g^{-1} \right)}^* \mu^j_k {\left( N^k E_j^{\top} \right)} \\
& = H^s - \sum_{\substack{i, j, k \\ \mathrm{cyc}}} {\left( g^{-1} \right)}^* \mu^j_k {\left( N^k E_j^{\top} - N^j E_k^{\top} \right)}.
\end{align*}
Here, \( E_i^{\times} = N^k E_j^{\top} - N^j E_k^{\top} \), which proves the former. The latter follows from the similar computation replacing the superscript \( \top \) with \( \times \).
\end{proof}

\noindent Applying Lemma~\ref{L:rotation} to Lemma~\ref{L:gauge-transformation}, we obtain the following.

\begin{pprop} \label{P:gauge-transformation}
Let \( g \colon M \to \mathrm{SO}(3) \) be a smooth map such that \( g = \boldsymbol{e}^{\theta \hat{e}} \) along \( S \) for some smooth maps \( \theta \colon S \to \mathbb{R} / 2 \boldsymbol{\pi} \mathbb{Z} \) and \( e \colon S \to \mathbb{S}^2 \). Then
\begin{align}
\begin{aligned}
H^{s \cdot g} & = H^s - e \cdot \Grad_{s^{\times}} \! {\left( \theta \right)} - \sin(\theta) \Div_{s^{\times}} (e) + {\left( 1 - \cos(\theta) \right)} \Curl_{s^{\times}} (e) \cdot e, \\
{\star} \tau^{s \cdot g} & = {\star} \tau^s - e \cdot \Grad_{s^{\top}} \! {\left( \theta \right)} - \sin(\theta) \Div_{s^{\top}} (e) + {\left( 1 - \cos(\theta) \right)} \Curl_{s^{\top}} (e) \cdot e.
\end{aligned}
\end{align}
\end{pprop}
\begin{proof}
Let the superscript \( \bullet \) in the following computation denote either \( \top \) or \( \times \). By Lemma~\ref{L:rotation}, for each \( (i,j,k) \in \{ (1,2,3), (2,3,1), (3,1,2) \} \), we have
\begin{align*}
& {\left( g^{-1} \right)}^* \mu^i_j {\left( E_k^{\bullet} \right)} = - \mathrm{d} \theta {\left( E_k^{\bullet} \right)} \hat{e}^i_j - \sin(\theta) \, \mathrm{d} \hat{e}^i_j {\left( E_k^{\bullet} \right)} + {\left( 1 - \cos(\theta) \right)} {\big( \widehat{\mathrm{d} e {\left( E_k^{\bullet} \right)} \times e} \big)}{}^i_j \\
& \qquad = \mathrm{d} \theta {\left( E_k^{\bullet} \right)} e^k + \sin(\theta) \, \mathrm{d} e^k {\left( E_k^{\bullet} \right)} - {\left( 1 - \cos(\theta) \right)} {\left( \mathrm{d} e {\left( E_k^{\bullet} \right)} \times e \right)}^k \\
& \qquad = E_k^{\bullet} {\left( \theta \right)} e^k + \sin(\theta) E_k^{\bullet} {\left( e^k \right)} - {\left( 1 - \cos(\theta) \right)} {\left( e^j E_k^{\bullet} {\left( e^i \right)} - e^i E_k^{\bullet} {\left( e^j \right)} \right)}.
\end{align*}
Here, we have
\begin{gather*}
\sum_{k = 1}^3 E_k^{\bullet} {\left( \theta \right)} e^k = e \cdot \Grad_{s^{\bullet}} \! {\left( \theta \right)}, \qquad \sum_{k = 1}^3 E_k^{\bullet} {\left( e^k \right)} = \Div_{s^{\bullet}} (e), \\
\sum_{\mathrm{cyc}} {\left( e^j E_k^{\bullet} {\left( e^i \right)} - e^i E_k^{\bullet} {\left( e^j \right)} \right)} = \sum_{\mathrm{cyc}} e^k {\left( E_i^{\bullet} {\left( e^j \right)} - E_j^{\bullet} {\left( e^i \right)} \right)} = \Curl_{s^{\bullet}} (e) \cdot e.
\end{gather*}
This completes the proof.
\end{proof}

We obtain several interesting corollaries.

\begin{pthrm} \label{T:gauge-transformation}
Let \( n \colon S \to \mathbb{S}^2 \) be the Gauss map of \( S \) with respect to \( s \). Let \( g \colon M \to \mathrm{SO}(3) \) be a smooth map such that \( g = \boldsymbol{e}^{\theta \hat{n}} \) along \( S \) for some smooth map \( \theta \colon S \to \mathbb{R} / 2 \boldsymbol{\pi} \mathbb{Z} \). Then
\begin{align}
H^{s \cdot g} & = \cos(\theta) H^s - \sin(\theta) {\star} \tau^s \quad \text{and} \quad {\star} \tau^{s \cdot g} = \sin(\theta) H^s + \cos(\theta) {\star} \tau^s.
\end{align}
That is, \( \boldsymbol{H} {\left( s \cdot \boldsymbol{e}^{\theta \hat{n}} \right)} = \boldsymbol{H}(s) \boldsymbol{e}^{\boldsymbol{i} \theta} \).
\end{pthrm}
\begin{proof}
If the superscript \( \bullet \) in the following denotes either \( \top \) or \( \times \), we have
\begin{align}
n \cdot \Grad_{s^{\bullet}} \! {\left( \theta \right)} & = \sum_{i = 1}^3 N^i E_i^{\bullet} {\left( \theta \right)} = N^{\bullet} {\left( \theta \right)} = 0.
\end{align}
We apply Corollary~\ref{C:divergence-curl} to Proposition~\ref{P:gauge-transformation}.
\end{proof}

\begin{pcor} \label{C:gauge-transformation}
If \( s \) and \( s'\) are two global oriented orthonormal smooth frames for \( M \) inducing the same Gauss map of \( S \), then\/ \( {\left| \boldsymbol{H}(s) \right|} = {\left| \boldsymbol{H}(s') \right|} \).
\end{pcor}
\begin{proof}
Let \( n \colon S \to \mathbb{S}^2 \) be the Gauss map of \( S \) with respect to either of them. Let \( g \colon M \to \mathrm{SO}(3) \) be the smooth map given by \( s' = s \cdot g \). Then \( gn = n \) on \( S \). This means that \( g \) is a spatial rotation about the axis \( n \) on \( S \).
\end{proof}

\subsection{Hopf differential} \label{SS:Hopf-differential}

In this subsection, we direct our attention to the relationship with the complex structure. We consider an oriented smooth surface \( S \) embedded in an oriented Riemann-Cartan \( 3 \)-manifold \( M \), where \( S \) is oriented by the unit normal vector field \( N \) along \( S \). Then \( S \) inherits a unique complex structure from its conformal structure. The general idea of the \emph{Hopf differential} extends to Riemann-Cartan geometry, which is a powerful tool for studying surfaces having constant mean curvature in minimal surface theory. For example, Fraser and Schoen~\cite{MR3404014} employed it to address the free boundary minimal disc problem.

\begin{pdef} \label{D:Hopf-differential}
The \emph{Hopf differential} is the quadratic differential on \( S \) defined by
\begin{align}
\varphi & = \II {\left( \frac{\partial}{\partial z}, \frac{\partial}{\partial z} \right)} \, \mathrm{d} z^2,
\end{align}
where \( z = x + \boldsymbol{i} y \) is a local complex chart on \( S \).
\end{pdef}

\begin{plem} \label{L:Hopf-differential}
Let \( z = x + \boldsymbol{i} y \) be a local complex chart on \( S \) such that the metric on \( S \) is of the form \( \lambda^2(z) \, \mathrm{d} z \, \mathrm{d} \bar{z} \) for some positive \( \lambda \). Extend\/ \( \frac{\partial}{\partial x} \),\/ \( \frac{\partial}{\partial y} \), and \( N \) arbitrarily smoothly to an open subset of \( M \). Then, along \( S \), we have
\begin{align}
{\left\langle N, \del_{\frac{\partial}{\partial x}} \frac{\partial}{\partial x} + \del_{\frac{\partial}{\partial y}} \frac{\partial}{\partial y} \right\rangle} & = \lambda^2 H \quad \text{and} \quad {\left\langle N, \del_{\frac{\partial}{\partial x}} \frac{\partial}{\partial y} - \del_{\frac{\partial}{\partial y}} \frac{\partial}{\partial x} \right\rangle} = \lambda^2 {\star} \tau.
\end{align}
\end{plem}
\begin{proof}
We have
\begin{align*}
\lambda^2 H & = \lambda^2 {\left( \II {\left( \frac{1}{\lambda} \frac{\partial}{\partial x}, \frac{1}{\lambda} \frac{\partial}{\partial x} \right)} + \II {\left( \frac{1}{\lambda} \frac{\partial}{\partial y}, \frac{1}{\lambda} \frac{\partial}{\partial y} \right)} \right)} = \II {\left( \frac{\partial}{\partial x}, \frac{\partial}{\partial x} \right)} + \II {\left( \frac{\partial}{\partial y}, \frac{\partial}{\partial y} \right)} \\
& = {\left\langle N, \del_{\frac{\partial}{\partial x}} \frac{\partial}{\partial x} \right\rangle} + {\left\langle N, \del_{\frac{\partial}{\partial y}} \frac{\partial}{\partial y} \right\rangle} = {\left\langle N, \del_{\frac{\partial}{\partial x}} \frac{\partial}{\partial x} + \del_{\frac{\partial}{\partial y}} \frac{\partial}{\partial y} \right\rangle}, \\
\lambda^2 {\star} \tau & = \lambda^2 {\left\langle N, T {\left( \frac{1}{\lambda} \frac{\partial}{\partial x}, \frac{1}{\lambda} \frac{\partial}{\partial y} \right)} \right\rangle} = {\left\langle N, T {\left( \frac{\partial}{\partial x}, \frac{\partial}{\partial y} \right)} \right\rangle} = {\left\langle N, \del_{\frac{\partial}{\partial x}} \frac{\partial}{\partial y} - \del_{\frac{\partial}{\partial y}} \frac{\partial}{\partial x} \right\rangle}
\end{align*}
along \( S \).
\end{proof}

\begin{plem}
Suppose the setup of Lemma~\ref{L:Hopf-differential}. Then, along \( S \), we have
\begin{align}
{\left\langle N, \del_{\frac{\partial}{\partial z}} \frac{\partial}{\partial \bar{z}} \right\rangle} & = \frac{\lambda^2}{4} \boldsymbol{H} \quad \text{and} \quad {\left\langle N, \del_{\frac{\partial}{\partial \bar{z}}} \frac{\partial}{\partial z} \right\rangle} = \frac{\lambda^2}{4} \bar{\boldsymbol{H}}.
\end{align}
\end{plem}
\begin{proof}
Along \( S \), we have
\begin{align*}
4 \del_{\frac{\partial}{\partial z}} \frac{\partial}{\partial \bar{z}} & = \del_{\frac{\partial}{\partial x} - \boldsymbol{i} \frac{\partial}{\partial y}} {\left( \frac{\partial}{\partial x} + \boldsymbol{i} \frac{\partial}{\partial y} \right)} \quad \text{and} \quad 4 \del_{\frac{\partial}{\partial \bar{z}}} \frac{\partial}{\partial z} = \del_{\frac{\partial}{\partial x} + \boldsymbol{i} \frac{\partial}{\partial y}} {\left( \frac{\partial}{\partial x} - \boldsymbol{i} \frac{\partial}{\partial y} \right)}.
\end{align*}
Expanding \( \mathbb{C} \)-linearly completes the proof.
\end{proof}

\begin{plem} \label{L:coefficients}
Suppose the setup of Lemma~\ref{L:Hopf-differential}. For any \( X \in \Gamma(TM) \), we have
\begin{align}
X & = \frac{2}{\lambda^2} {\left\langle \frac{\partial}{\partial \bar{z}}, X \right\rangle} \frac{\partial}{\partial z} + \frac{2}{\lambda^2} {\left\langle \frac{\partial}{\partial z}, X \right\rangle} \frac{\partial}{\partial \bar{z}} + {\left\langle N, X \right\rangle} N \quad \text{along} \ S.
\end{align}
\end{plem}
\begin{proof}
Along \( S \), we have \( {\left\langle \frac{\partial}{\partial z}, \frac{\partial}{\partial z} \right\rangle} = {\left\langle \frac{\partial}{\partial \bar{z}}, \frac{\partial}{\partial \bar{z}} \right\rangle} = 0 \) and \( {\left\langle \frac{\partial}{\partial z}, \frac{\partial}{\partial \bar{z}} \right\rangle} = \lambda^2 / 2 \).
\end{proof}

\begin{plem} \label{L:derivatives}
Suppose the setup of Lemma~\ref{L:Hopf-differential}. Then, along \( S \), we have
\begin{align}
\del_{\frac{\partial}{\partial z}} \frac{\partial}{\partial \bar{z}} & = \frac{2}{\lambda^2} {\left\langle \frac{\partial}{\partial z}, T_S {\left( \frac{\partial}{\partial z}, \frac{\partial}{\partial \bar{z}} \right)} \right\rangle} \frac{\partial}{\partial \bar{z}} + \frac{\lambda^2}{4} \boldsymbol{H} N, \\
\del_{\frac{\partial}{\partial \bar{z}}} \frac{\partial}{\partial z} & = \frac{2}{\lambda^2} {\left\langle \frac{\partial}{\partial \bar{z}}, T_S {\left( \frac{\partial}{\partial \bar{z}}, \frac{\partial}{\partial z} \right)} \right\rangle} \frac{\partial}{\partial z} + \frac{\lambda^2}{4} \bar{\boldsymbol{H}} N, \\
\del_{\frac{\partial}{\partial z}} N & = - \frac{2}{\lambda^2} \II {\left( \frac{\partial}{\partial z}, \frac{\partial}{\partial z} \right)} \frac{\partial}{\partial \bar{z}} - \frac{1}{2} \boldsymbol{H} \frac{\partial}{\partial z}, \\
\del_{\frac{\partial}{\partial \bar{z}}} N & = - \frac{2}{\lambda^2} \II {\left( \frac{\partial}{\partial \bar{z}}, \frac{\partial}{\partial \bar{z}} \right)} \frac{\partial}{\partial z} - \frac{1}{2} \bar{\boldsymbol{H}} \frac{\partial}{\partial \bar{z}},
\end{align}
where \( T_S \) is the torsion of \( S \).
\end{plem}
\begin{proof}
We use Lemma~\ref{L:coefficients}.
\end{proof}

Now, we are ready to present the key equality in this subsection.

\begin{pprop} \label{P:Hopf-differential}
Suppose the setup of Lemma~\ref{L:Hopf-differential}. Then, along \( S \), we have
\begin{align}
\begin{aligned}
\frac{\partial}{\partial \bar{z}} \II {\left( \frac{\partial}{\partial z}, \frac{\partial}{\partial z} \right)} & = \begin{aligned}[t]
& \frac{\lambda^2}{4} \overline{{\left( \frac{\partial \boldsymbol{H}}{\partial \bar{z}} \right)}} - \frac{\boldsymbol{i}}{2} R {\left( \frac{\partial}{\partial x}, \frac{\partial}{\partial y}, \frac{\partial}{\partial z}, N \right)} \\
& - \frac{1}{2} \II {\left( J_S T_S {\left( \frac{\partial}{\partial x}, \frac{\partial}{\partial y} \right)}, \frac{\partial}{\partial z} \right)},
\end{aligned}
\end{aligned}
\end{align}
where \( R \) is the Riemann curvature of \( M \), \( J_S \) is the almost complex structure on \( S \), and \( T_S \) is the torsion of \( S \).
\end{pprop}
\begin{proof}
Along \( S \), we have
\begin{align*}
& \frac{\partial}{\partial \bar{z}} \II {\left( \frac{\partial}{\partial z}, \frac{\partial}{\partial z} \right)} = \frac{\partial}{\partial \bar{z}} {\left\langle \del_{\frac{\partial}{\partial z}} \frac{\partial}{\partial z}, N \right\rangle} = {\left\langle \del_{\frac{\partial}{\partial \bar{z}}} \del_{\frac{\partial}{\partial z}} \frac{\partial}{\partial z}, N \right\rangle} + {\left\langle \del_{\frac{\partial}{\partial z}} \frac{\partial}{\partial z}, \del_{\frac{\partial}{\partial \bar{z}}} N \right\rangle} \\
& \qquad = \begin{aligned}[t]
& {\left\langle R {\left( \frac{\partial}{\partial \bar{z}}, \frac{\partial}{\partial z} \right)} \frac{\partial}{\partial z} + \del_{\frac{\partial}{\partial z}} \del_{\frac{\partial}{\partial \bar{z}}} \frac{\partial}{\partial z}, N \right\rangle} \\
& + {\left\langle \del_{\frac{\partial}{\partial z}} \frac{\partial}{\partial z}, - \frac{2}{\lambda^2} \II {\left( \frac{\partial}{\partial \bar{z}}, \frac{\partial}{\partial \bar{z}} \right)} \frac{\partial}{\partial z} - \frac{1}{2} \bar{\boldsymbol{H}} \frac{\partial}{\partial \bar{z}} \right\rangle}
\end{aligned} \\
& \qquad = \begin{aligned}[t]
& R {\left( \frac{\partial}{\partial \bar{z}}, \frac{\partial}{\partial z}, \frac{\partial}{\partial z}, N \right)} + {\left\langle \del_{\frac{\partial}{\partial z}} {\left( \frac{2}{\lambda^2} {\left\langle \frac{\partial}{\partial \bar{z}}, T_S {\left( \frac{\partial}{\partial \bar{z}}, \frac{\partial}{\partial z} \right)} \right\rangle} \frac{\partial}{\partial z} + \frac{\lambda^2}{4} \bar{\boldsymbol{H}} N \right)}, N \right\rangle} \\
& - \frac{2}{\lambda^2} \II {\left( \frac{\partial}{\partial \bar{z}}, \frac{\partial}{\partial \bar{z}} \right)} {\left\langle \del_{\frac{\partial}{\partial z}} \frac{\partial}{\partial z}, \frac{\partial}{\partial z} \right\rangle} - \frac{1}{2} \bar{\boldsymbol{H}} {\left\langle \del_{\frac{\partial}{\partial z}} \frac{\partial}{\partial z}, \frac{\partial}{\partial \bar{z}} \right\rangle}.
\end{aligned}
\end{align*}
Let us look at term by term. For the first term, we have
\begin{align*}
& R {\left( \frac{\partial}{\partial \bar{z}}, \frac{\partial}{\partial z}, \frac{\partial}{\partial z}, N \right)} = \frac{1}{2} {\left( - R {\left( \frac{\partial}{\partial z}, \frac{\partial}{\partial \bar{z}}, \frac{\partial}{\partial z}, N \right)} + R {\left( \frac{\partial}{\partial \bar{z}}, \frac{\partial}{\partial z}, \frac{\partial}{\partial z}, N \right)} \right)} \\
& \qquad = \frac{1}{2} R {\left( \frac{\partial}{\partial z} + \frac{\partial}{\partial \bar{z}}, \frac{\partial}{\partial z} - \frac{\partial}{\partial \bar{z}}, \frac{\partial}{\partial z}, N \right)} = \frac{1}{2} R {\left( \frac{\partial}{\partial x}, - \boldsymbol{i} \frac{\partial}{\partial y}, \frac{\partial}{\partial z}, N \right)}.
\end{align*}
For the second term, we have
\begin{align*}
& {\left\langle \del_{\frac{\partial}{\partial z}} {\left( \frac{2}{\lambda^2} {\left\langle \frac{\partial}{\partial \bar{z}}, T_S {\left( \frac{\partial}{\partial \bar{z}}, \frac{\partial}{\partial z} \right)} \right\rangle} \frac{\partial}{\partial z} + \frac{\lambda^2}{4} \bar{\boldsymbol{H}} N \right)}, N \right\rangle} \\
& \qquad = \begin{aligned}[t]
& \frac{\partial}{\partial z} {\left( \frac{2}{\lambda^2} {\left\langle \frac{\partial}{\partial \bar{z}}, T_S {\left( \frac{\partial}{\partial \bar{z}}, \frac{\partial}{\partial z} \right)} \right\rangle} \right)} {\left\langle \frac{\partial}{\partial z}, N \right\rangle} \\
& + \frac{2}{\lambda^2} {\left\langle \frac{\partial}{\partial \bar{z}}, T_S {\left( \frac{\partial}{\partial \bar{z}}, \frac{\partial}{\partial z} \right)} \right\rangle} {\left\langle \del_{\frac{\partial}{\partial z}} \frac{\partial}{\partial z}, N \right\rangle} \\
& + \frac{\partial}{\partial z} {\left( \frac{\lambda^2}{4} \bar{\boldsymbol{H}} \right)} {\left\langle N, N \right\rangle} + \frac{\lambda^2}{4} \bar{\boldsymbol{H}} {\left\langle \del_{\frac{\partial}{\partial z}} N, N \right\rangle}
\end{aligned} \\
& \qquad = \frac{2}{\lambda^2} {\left\langle \frac{\partial}{\partial \bar{z}}, T_S {\left( \frac{\partial}{\partial \bar{z}}, \frac{\partial}{\partial z} \right)} \right\rangle} {\left\langle \del_{\frac{\partial}{\partial z}} \frac{\partial}{\partial z}, N \right\rangle} + \bar{\boldsymbol{H}} \frac{\partial}{\partial z} \frac{\lambda^2}{4} + \frac{\lambda^2}{4} \frac{\partial \bar{\boldsymbol{H}}}{\partial z}.
\end{align*}
The third term vanishes, since \( {\left\langle \frac{\partial}{\partial z}, \frac{\partial}{\partial z} \right\rangle} = 0 \). For the last term, we have
\begin{align*}
- \frac{1}{2} \bar{\boldsymbol{H}} {\left\langle \del_{\frac{\partial}{\partial z}} \frac{\partial}{\partial z}, \frac{\partial}{\partial \bar{z}} \right\rangle} & = - \frac{1}{2} \bar{\boldsymbol{H}} {\left( \frac{\partial}{\partial z} {\left\langle \frac{\partial}{\partial z}, \frac{\partial}{\partial \bar{z}} \right\rangle} - {\left\langle \frac{\partial}{\partial z}, \del_{\frac{\partial}{\partial z}} \frac{\partial}{\partial \bar{z}} \right\rangle} \right)} \\
& = - \bar{\boldsymbol{H}} \frac{\partial}{\partial z} \frac{\lambda^2}{4} + \frac{2}{\lambda^2} {\left\langle \del_{\frac{\partial}{\partial \bar{z}}} \frac{\partial}{\partial z}, N \right\rangle} {\left\langle \frac{\partial}{\partial z}, T_S {\left( \frac{\partial}{\partial z}, \frac{\partial}{\partial \bar{z}} \right)} \right\rangle}.
\end{align*}
To sum up, we obtain
\begin{align*}
\frac{\partial}{\partial \bar{z}} \II {\left( \frac{\partial}{\partial z}, \frac{\partial}{\partial z} \right)} & = \frac{\lambda^2}{4} \frac{\partial \bar{\boldsymbol{H}}}{\partial z} - \frac{\boldsymbol{i}}{2} R {\left( \frac{\partial}{\partial x}, \frac{\partial}{\partial y}, \frac{\partial}{\partial z}, N \right)} + {\left\langle \del_X \frac{\partial}{\partial z}, N \right\rangle},
\end{align*}
where
\begin{align*}
X & = \frac{2}{\lambda^2} {\left\langle \frac{\partial}{\partial \bar{z}}, T_S {\left( \frac{\partial}{\partial \bar{z}}, \frac{\partial}{\partial z} \right)} \right\rangle} \frac{\partial}{\partial z} + \frac{2}{\lambda^2} {\left\langle \frac{\partial}{\partial z}, T_S {\left( \frac{\partial}{\partial z}, \frac{\partial}{\partial \bar{z}} \right)} \right\rangle} \frac{\partial}{\partial \bar{z}} \\
& = \frac{2}{\lambda^2} {\left\langle - \boldsymbol{i} J_S \frac{\partial}{\partial \bar{z}}, T_S {\left( \frac{\partial}{\partial z}, \frac{\partial}{\partial \bar{z}} \right)} \right\rangle} \frac{\partial}{\partial z} + \frac{2}{\lambda^2} {\left\langle - \boldsymbol{i} J_S \frac{\partial}{\partial z}, T_S {\left( \frac{\partial}{\partial z}, \frac{\partial}{\partial \bar{z}} \right)} \right\rangle} \frac{\partial}{\partial \bar{z}} \\
& = \frac{2}{\lambda^2} {\left\langle \frac{\partial}{\partial \bar{z}}, \boldsymbol{i} J_S T_S {\left( \frac{\partial}{\partial z}, \frac{\partial}{\partial \bar{z}} \right)} \right\rangle} \frac{\partial}{\partial z} + \frac{2}{\lambda^2} {\left\langle \frac{\partial}{\partial z}, \boldsymbol{i} J_S T_S {\left( \frac{\partial}{\partial z}, \frac{\partial}{\partial \bar{z}} \right)} \right\rangle} \frac{\partial}{\partial \bar{z}}.
\end{align*}
By Lemma~\ref{L:coefficients}, we obtain
\begin{align*}
X & = \boldsymbol{i} J_S T_S {\left( \frac{\partial}{\partial z}, \frac{\partial}{\partial \bar{z}} \right)} = \frac{\boldsymbol{i}}{4} J_S T_S {\left( \frac{\partial}{\partial x} - \boldsymbol{i} \frac{\partial}{\partial y}, \frac{\partial}{\partial x} + \boldsymbol{i} \frac{\partial}{\partial y} \right)} \\
& = \frac{\boldsymbol{i}}{4} {\left( \boldsymbol{i} J_S T_S {\left( \frac{\partial}{\partial x}, \frac{\partial}{\partial y} \right)} - \boldsymbol{i} J_S T_S {\left( \frac{\partial}{\partial y}, \frac{\partial}{\partial x} \right)} \right)} = - \frac{1}{2} J_S T_S {\left( \frac{\partial}{\partial x}, \frac{\partial}{\partial y} \right)}.
\end{align*}
This completes the proof.
\end{proof}

\begin{pcor} \label{C:Hopf-differential}
Suppose that the following skew-symmetric smooth\/ \( (1,2) \)-tensor field \( L \) on \( S \) vanishes everywhere.
\begin{align}
L(X,Y) & = R {\left( X, Y \right)} N - J_S W J_S T_S {\left( X, Y \right)} \qquad (X, Y \in \Gamma(TS)),
\end{align}
where \( R \) is the curvature of \( M \), \( J_S \) is the almost complex structure on \( S \), and \( T_S \) is the torsion of \( S \). Then \( \boldsymbol{H} \) is holomorphic on \( S \) iff the Hopf differential \( \varphi \) on \( S \) is a holomorphic quadratic differential.
\end{pcor}
\begin{proof}
Look at the equality in Proposition~\ref{P:Hopf-differential}. We have
\begin{align*}
& - \frac{\boldsymbol{i}}{2} R {\left( \frac{\partial}{\partial x}, \frac{\partial}{\partial y}, \frac{\partial}{\partial z}, N \right)} = \frac{\boldsymbol{i}}{2} R {\left( \frac{\partial}{\partial x}, \frac{\partial}{\partial y}, N, \frac{\partial}{\partial z} \right)} = \frac{\boldsymbol{i}}{2} {\left\langle R {\left( \frac{\partial}{\partial x}, \frac{\partial}{\partial y} \right)} N, \frac{\partial}{\partial z} \right\rangle} \\
& \qquad = \frac{1}{2} {\left\langle R {\left( \frac{\partial}{\partial x}, \frac{\partial}{\partial y} \right)} N, J_S \frac{\partial}{\partial z} \right\rangle} = - \frac{1}{2} {\left\langle J_S R {\left( \frac{\partial}{\partial x}, \frac{\partial}{\partial y} \right)} N, \frac{\partial}{\partial z} \right\rangle}, \\
& - \frac{1}{2} \II {\left( J_S T_S {\left( \frac{\partial}{\partial x}, \frac{\partial}{\partial y} \right)}, \frac{\partial}{\partial z} \right)} = - \frac{1}{2} {\left\langle W J_S T_S {\left( \frac{\partial}{\partial x}, \frac{\partial}{\partial y} \right)}, \frac{\partial}{\partial z} \right\rangle}.
\end{align*}
Therefore, their summation equals \( - \frac{1}{2} {\big\langle J_S L {\big( \frac{\partial}{\partial x}, \frac{\partial}{\partial y} \big)}, \frac{\partial}{\partial z} \big\rangle} \), which vanishes.
\end{proof}

In Corollary~\ref{C:Hopf-differential}, the assumption is imposed on \( S \), rather than on \( M \). One could think that it is more relevant to assume a property on \( M \). To find a sufficient condition on \( M \) so that the assumption in Corollary~\ref{C:Hopf-differential} is satisfied for every \( S \), recall that \( J_S X = N \times X \) for all \( X \in \Gamma(TS) \), so that
\begin{align}
L(X,Y) & = R {\left( X, Y \right)} N + N \times \del_{N \times T {\left( X, Y \right)}} N
\end{align}
for all \( X, Y \in \Gamma(TS) \) by Propositions~\ref{P:SSFF} and~\ref{P:Weingarten-equation}. Therefore, we want
\begin{align}
R {\left( X, Y \right)} Z + Z \times \del_{Z \times T {\left( X, Y \right)}} Z & = 0
\end{align}
for every local oriented orthonormal smooth frame \( {\left( X, Y, Z \right)} \) for \( M \), which holds iff
\begin{align*}
{\left\{ \begin{aligned}
R {\left( X, Y, Z, X \right)} + {\left\langle Z \times \del_{Z \times T {\left( X, Y \right)}} Z, X \right\rangle} & = 0, \\
R {\left( X, Y, Z, Y \right)} + {\left\langle Z \times \del_{Z \times T {\left( X, Y \right)}} Z, Y \right\rangle} & = 0
\end{aligned} \right.}
\end{align*}
iff \( \Ric(Y,Z) = {\left\langle \del_{Z \times T {\left( X, Y \right)}} Z, Y \right\rangle} \) and \( \Ric(X,Z) = {\left\langle \del_{Z \times T {\left( X, Y \right)}} Z, X \right\rangle} \) by elementary identities of the cross product. Equivalently,
\begin{align}
\Ric(X,Y) & = {\left\langle X, \del_{Y \times T {\left( X \times Y, X \right)}} Y \right\rangle}
\end{align}
for any orthogonal pair of local unit smooth vector fields \( X \) and \( Y \) on \( M \). The left-hand side is a smooth tensor, while the right-hand side depends on the local behavior of \( Y \). Therefore, we would expect
\begin{align}
\Ric(X,Y) & = 0 \quad \text{and} \quad T(X,Y) \perp X, Y
\end{align}
for every orthogonal pair of unit vectors \( X \) and \( Y \) (at the same point). The former is equivalent to the Ricci curvature being proportional to the metric at each point, since \( 0 = 2 \, \Ric {\big( \frac{X + Y}{\sqrt{2}}, \frac{X - Y}{\sqrt{2}} \big)} = \Ric(X,X) - \Ric(Y,Y) \). The latter is equivalent to the torsion being proportional to the cross product at each point.

\begin{pprop} \label{P:sufficient-condition}
Suppose that the Ricci curvature and the torsion of \( M \) are respectively proportional to the metric and the cross product at every point. Then, for any oriented smooth surface \( S \) embedded in \( M \), \( \boldsymbol{H} \) is holomorphic on \( S \) iff the Hopf differential \( \varphi \) on \( S \) is a holomorphic quadratic differential.
\end{pprop}
\begin{proof}
Let \( {\left( X, Y \right)} \) be a local oriented orthonormal smooth frame for \( S \). Then
\begin{align*}
R {\left( X, Y \right)} N & = R {\left( X, Y, N, X \right)} X + R {\left( X, Y, N, Y \right)} Y = \Ric(Y,N) X - \Ric(X,N) Y = 0
\end{align*}
and \( T_S {\left( X, Y \right)} = T {\left( X, Y \right)}{}^{\top} = 0 \) by Proposition~\ref{P:SSFF}. Therefore, \( L \) vanishes.
\end{proof}

We illustrate several examples that satisfy the condition in Proposition~\ref{P:sufficient-condition}.

\begin{pex} \label{Ex:Einstein}
If \( M \) is an oriented Riemannian \( 3 \)-manifold, then the condition in Proposition~\ref{P:sufficient-condition} exactly means that \( M \) is an Einstein \( 3 \)-manifold, which is equivalent to having constant sectional curvature in dimension three. Therefore, for every oriented smooth surface \( S \) embedded in an oriented Riemannian \( 3 \)-manifold \( M \) having constant sectional curvature, \( S \) has constant mean curvature iff the Hopf differential \( \varphi \) on \( S \) is a holomorphic quadratic differential. This is a classical result (Lemma~2.2 in Chapter~VI of~\cite{MR1013786}).
\end{pex}

\begin{pex}
Let \( M = \mathrm{SO}(3) \) be the Weitzenb{\"o}ck \( 3 \)-manifold with the left-invariant global smooth frame \( {\left( L_1, L_2, L_3 \right)} \) in~\eqref{E:L}. By Proposition~\ref{P:Weitzenbock}, the torsion satisfies \( T {\left( L_i, L_j \right)} = - {\left[ L_i, L_j \right]} = - L_k = - L_i \times L_j \) for all cyclic triple \( {\left( i, j, k \right)} \). That is, \( T {\left( X, Y \right)} = - X \times Y \) for all \( X, Y \in \Gamma(TM) \). Therefore, \( M \) satisfies the condition in Proposition~\ref{P:sufficient-condition}.
\end{pex}

\begin{pex} \label{Ex:Cartan-Schouten}
Let \( M = \mathbb{R}^3 \) be the Riemann-Cartan \( 3 \)-manifold whose metric is the standard Euclidean metric but connection is the one determined by
\begin{align}
\del_{\partial_i} \partial_j & = \lambda \, \partial_i \times \partial_j \qquad (i, j \in \{ 1, 2, 3 \})
\end{align}
and the Leibniz rule, for some fixed \( \lambda \in \mathbb{R} \), where \( {\left( \partial_1, \partial_2, \partial_3 \right)} \) is the standard global smooth frame. This connection is motivated by the Cartan-Schouten connections on Lie groups. It is indeed metric-compatible: For all \( i, j, k \in \{ 1, 2, 3 \} \),
\begin{align*}
\partial_i {\left\langle \partial_j, \partial_k \right\rangle} - {\left\langle \del_{\partial_i} \partial_j, \partial_k \right\rangle} - {\left\langle \partial_j, \del_{\partial_i} \partial_k \right\rangle} & = - \lambda {\left( {\left\langle \partial_i \times \partial_j, \partial_k \right\rangle} + {\left\langle \partial_j, \partial_i \times \partial_k \right\rangle} \right)} = 0.
\end{align*}
Unless \( \lambda = 0 \), \( M \) is neither flat nor torsion-free: For all \( i, j, k \in \{ 1, 2, 3 \} \),
\begin{align}
T {\left( \partial_i, \partial_j \right)} & = \del_{\partial_i} \partial_j - \del_{\partial_j} \partial_i = 2 \lambda \, \partial_i \times \partial_j, \\
R {\left( \partial_i, \partial_j \right)} \partial_k & = \del_{\partial_i} \del_{\partial_j} \partial_k - \del_{\partial_j} \del_{\partial_i} \partial_k = \lambda^2 {\left( \partial_i \times \partial_j \right)} \times \partial_k,
\end{align}
i.e., \( T {\left( X, Y \right)} = 2 \lambda X \times Y \) and \( R {\left( X, Y \right)} Z = \lambda^2 {\left( X \times Y \right)} \times Z \) for all \( X, Y, Z \in \Gamma(TM) \). In particular, \( \Ric = - 2 \lambda^2 {\left\langle {}\cdot{}, {}\cdot{} \right\rangle} \), and \( M \) has constant sectional curvature \( - \lambda^2 \). Therefore, \( M \) satisfies the condition in Proposition~\ref{P:sufficient-condition}.
\end{pex}

Meanwhile, the third fundamental form also yields a quadratic differential in a way similar to the Hopf differential (Definition~\ref{D:Hopf-differential}).

\begin{pdef} \label{D:TFF}
We define a quadratic differential \( \psi \) on \( S \) by
\begin{align}
\psi & = \III {\left( \frac{\partial}{\partial z}, \frac{\partial}{\partial z} \right)} \, \mathrm{d} z^2,
\end{align}
where \( z = x + \boldsymbol{i} y \) is a local complex chart on \( S \).
\end{pdef}

\begin{pprop} \label{P:TFF}
We have \( \psi = \boldsymbol{H} \varphi \).
\end{pprop}
\begin{proof}
Let \( z = x + \boldsymbol{i} y \) be a local complex chart on \( S \) such that the metric on \( S \) is of the form \( \lambda^2(z) \, \mathrm{d} z \, \mathrm{d} \bar{z} \) for some positive \( \lambda \). By Lemma~\ref{L:derivatives}, we have
\begin{align*}
\III {\left( \frac{\partial}{\partial z}, \frac{\partial}{\partial z} \right)} & = {\left\langle \del_{\frac{\partial}{\partial z}} N, \del_{\frac{\partial}{\partial z}} N \right\rangle} = \lambda^2 {\left( - \frac{2}{\lambda^2} \II {\left( \frac{\partial}{\partial z}, \frac{\partial}{\partial z} \right)} \right)} {\left( - \frac{1}{2} \boldsymbol{H} \right)},
\end{align*}
which completes the proof.
\end{proof}

\begin{pcor} \label{C:TFF}
Suppose that the skew-symmetric smooth\/ \( (1,2) \)-tensor field \( L \) on \( S \) in Corollary~\ref{C:Hopf-differential} vanishes everywhere. If \( \boldsymbol{H} \) is holomorphic on \( S \), then \( \psi \) is a holomorphic quadratic differential on \( S \).
\end{pcor}

\subsection{Minimal surfaces} \label{SS:minimal}

In this subsection, we discuss minimal surfaces in the context of Riemann-Cartan geometry. We will reveal a close relationship between the conformality of the Gauss map and the condition \( \boldsymbol{H} = 0 \). This extends the classical fact that minimal surfaces in \( \mathbb{R}^3 \) are completely characterized by the conformality of the Gauss map. We start with introducing several geometric concepts in Riemann-Cartan geometry.

\begin{pdef} \label{D:minimal}
Let \( S \) be an oriented smooth surface embedded in an oriented Riemann-Cartan \( 3 \)-manifold.
\begin{enumerate}
\item \( S \) is said to be \emph{minimal} iff \( \boldsymbol{H} = 0 \) everywhere.
\item \( S \) is said to be \emph{umbilic} at \( x \in S \) iff the Weingarten map of \( S \) is given by
\begin{align}
W & = \frac{1}{2} \begin{pmatrix*}[r]
H & - {\star} \tau \\
{\star} \tau & H
\end{pmatrix*} \quad \text{at} \ x \label{E:umbilic}
\end{align}
relative to some (equivalently, any) oriented orthonormal basis for \( T_x S \).
\item \( S \) is said to be \emph{totally umbilic} iff it is umbilic everywhere.
\end{enumerate}
\end{pdef}

\noindent In Definition~\ref{D:minimal}, when the ambient \( 3 \)-manifold is torsion-free (i.e., the Levi-Civita case), then all notions correspond to the usual definitions in Riemannian geometry. Also, it is worth pointing out that~\eqref{E:umbilic} is expressed as \( W = \frac{1}{2} \boldsymbol{H} \) under the matrix representation of complex numbers.

The Hopf differential \( \varphi \) measures umbilicity in the following sense.

\begin{pprop} \label{P:umbilic}
Let \( S \) be an oriented smooth surface embedded in an oriented Riemann-Cartan\/ \( 3 \)-manifold. Let \( \varphi \) be the Hopf differential of \( S \). Then the following two hold.
\begin{enumerate}
\item (Local version) \( \varphi = 0 \) at a point in \( S \) iff \( S \) is umbilic at the point.
\item (Global version) \( \varphi \) vanishes everywhere iff \( S \) is totally umbilic.
\end{enumerate}
\end{pprop}
\begin{proof}
We only need to prove~(a), since~(b) is immediate from~(a). Fix a point in \( S \) and consider locally near the point. Let \( z = x + \boldsymbol{i} y \) be a local complex chart on \( S \) such that the metric on \( S \) is of the form \( \lambda^2(z) \, \mathrm{d} z \, \mathrm{d} \bar{z} \) for some positive \( \lambda \). Then
\begin{align*}
4 \II {\left( \frac{\partial}{\partial z}, \frac{\partial}{\partial z} \right)} & = {\left( \II {\left( \frac{\partial}{\partial x}, \frac{\partial}{\partial x} \right)} - \II {\left( \frac{\partial}{\partial y}, \frac{\partial}{\partial y} \right)} \right)} - \boldsymbol{i} {\left( \II {\left( \frac{\partial}{\partial x}, \frac{\partial}{\partial y} \right)} + \II {\left( \frac{\partial}{\partial y}, \frac{\partial}{\partial x} \right)} \right)} \\
& = \lambda^2 {\left( W^1_1 - W^2_2 \right)} - \lambda^2 \boldsymbol{i} {\left( W^2_1 + W^1_2 \right)}
\end{align*}
relative to the local oriented orthonormal smooth frame \( {\big( \frac{1}{\lambda} \frac{\partial}{\partial x}, \frac{1}{\lambda} \frac{\partial}{\partial y} \big)} \) for \( S \). This implies that \( \varphi = 0 \) at the point iff \( W^1_1 = W^2_2 \) and \( W^2_1 = - W^1_2 \). Recall that \( H = W^1_1 + W^2_2 \) and \( {\star} \tau = W^2_1 - W^1_2 \). This completes the proof.
\end{proof}

For the quadratic differential \( \psi \) in Definition~\ref{D:TFF}, we have the following.

\begin{pprop} \label{P:TFF-minimal}
Let \( S \) be an oriented smooth surface embedded in an oriented Riemann-Cartan\/ \( 3 \)-manifold. Let \( \psi \) be the quadratic differential on \( S \) in Definition~\ref{D:TFF}. Then the following two hold.
\begin{enumerate}
\item (Local version) \( \psi = 0 \) at a point in \( S \) iff \( \boldsymbol{H} = 0 \) or \( S \) is umbilic at the point.
\item (Global version) If \( S \) is connected and nowhere geodesic (i.e.,\/ \( \II \ne 0 \) everywhere), then \( \psi \) vanishes everywhere iff \( S \) is minimal or totally umbilic.
\end{enumerate}
\end{pprop}
\begin{proof}
(a) follows from Propositions~\ref{P:TFF} and~\ref{P:umbilic}. The ``if'' direction of~(b) is immediate from~(a). Consider the ``only if'' direction of~(b). As subsets of \( S \),
\begin{align*}
S & = {\left\{ \psi = 0 \right\}} = {\left\{ \boldsymbol{H} = 0 \right\}} \cup {\left\{ \varphi = 0 \right\}} \quad \text{and} \quad \varnothing = {\left\{ \II = 0 \right\}} = {\left\{ \boldsymbol{H} = 0 \right\}} \cap {\left\{ \varphi = 0 \right\}}.
\end{align*}
The sets \( {\left\{ \boldsymbol{H} = 0 \right\}} \) and \( {\left\{ \varphi = 0 \right\}} \) are closed in \( S \). They form a disjoint open covering of \( S \). As \( S \) is connected, one of them must be empty.
\end{proof}

In fact, the quadratic differential \( \psi \) in Definition~\ref{D:TFF} is closely related to the conformality of the Gauss map. It would be a good habit to state the definition of conformality that this paper follows, since there are a lot of different conventions.

\begin{pdef}
A smooth map \( F \colon {\left( \Sigma_1, g_1 \right)} \to {\left( \Sigma_2, g_2 \right)} \) between Riemannian manifolds is said to be \emph{conformal} or \emph{angle-preserving} at \( x \in \Sigma_1 \) iff \( F^* g_2 = k g_1 \) at \( x \) for some \( k > 0 \). It is said to be \emph{conformal} iff it is conformal everywhere.
\end{pdef}

\noindent Note that we do not require a conformal map to be a diffeomorphism. The following provides a simple criterion for the conformality of smooth maps between surfaces.

\begin{plem} \label{L:conformal}
A smooth map \( F \colon \Sigma_1 \to \Sigma_2 \) between Riemannian surfaces is conformal at a point in\/ \( \Sigma_1 \) iff
\begin{align}
{\left\langle \frac{\partial F}{\partial x}, \frac{\partial F}{\partial x} \right\rangle}_{\Sigma_2} & = {\left\langle \frac{\partial F}{\partial y}, \frac{\partial F}{\partial y} \right\rangle}_{\Sigma_2} > 0 \quad \text{and} \quad {\left\langle \frac{\partial F}{\partial x}, \frac{\partial F}{\partial y} \right\rangle}_{\Sigma_2} = 0
\end{align}
at the point, where \( z = x + \boldsymbol{i} y \) is a local complex chart on\/ \( \Sigma_1 \) such that the metric on\/ \( \Sigma_1 \) is of the form \( \lambda^2(z) \, \mathrm{d} z \, \mathrm{d} \bar{z} \) for some positive \( \lambda \).
\end{plem}
\begin{proof}
The ``only if'' direction is straightforward. For the ``if'' direction, we claim that \( {\left\langle F_* X, F_* Y \right\rangle}_{\Sigma_2} = k {\left\langle X, Y \right\rangle}_{\Sigma_1} \) for all vectors \( X \) and \( Y \) at the point, where
\begin{align}
k & = \frac{1}{\lambda^2} {\left\langle \frac{\partial F}{\partial x}, \frac{\partial F}{\partial x} \right\rangle}_{\Sigma_2} = \frac{1}{\lambda^2} {\left\langle \frac{\partial F}{\partial y}, \frac{\partial F}{\partial y} \right\rangle}_{\Sigma_2} > 0.
\end{align}
Due to the \( \mathbb{R} \)-bilinearity of the desired equality, it suffices to check the four cases \( X, Y \in {\big\{ \frac{\partial}{\partial x}, \frac{\partial}{\partial y} \big\}} \), but these cases are obvious.
\end{proof}

\begin{pprop} \label{P:TFF-conformal}
Let \( S \) be an oriented smooth surface embedded in an oriented flat Riemann-Cartan\/ \( 3 \)-manifold \( M \). Let \( \pi \colon FM \surjto M \) be the oriented orthonormal frame bundle, which has the flat connection \( \omega \in \Omega^1 {\left( FM, \mathfrak{so}(3) \right)} \). Let \( \psi \) be the quadratic differential on \( S \) in Definition~\ref{D:TFF}. Then the following two hold.
\begin{enumerate}
\item (Local version) Let \( s \) be a local \( \omega \)-horizontal smooth section of \( \pi \colon FM \surjto M \) near a point in \( S \). Then the local Gauss map \( n^s \) is conformal at the point iff \( \psi = 0 \) and\/ \( \II \ne 0 \) at the point.
\item (Global version) The local Gauss map with respect to any local \( \omega \)-horizontal smooth section of \( \pi \colon FM \surjto M \) is conformal iff \( \psi \) vanishes everywhere and \( S \) is nowhere geodesic (i.e.,\/ \( \II \ne 0 \) everywhere).
\end{enumerate}
\end{pprop}
\begin{proof}
We only need to prove~(a), since~(b) is immediate from~(a). Fix a point in \( S \) and consider locally near the point. Let \( z = x + \boldsymbol{i} y \) be a local complex chart on \( S \) such that the metric on \( S \) is of the form \( \lambda^2(z) \, \mathrm{d} z \, \mathrm{d} \bar{z} \) for some positive \( \lambda \). Let \( s = {\left( E_1, E_2, E_3 \right)} \). Then \( n^s \) is conformal at the point iff
\begin{align*}
\frac{\partial n^s}{\partial x} \cdot \frac{\partial n^s}{\partial x} & = \frac{\partial n^s}{\partial y} \cdot \frac{\partial n^s}{\partial y} > 0 \quad \text{and} \quad \frac{\partial n^s}{\partial x} \cdot \frac{\partial n^s}{\partial y} = 0
\end{align*}
at the point (where the dot product is taken in \( \mathbb{S}^2 \subseteq \mathbb{R}^3 \)) iff
\begin{align*}
{\left\langle \frac{\partial N^i}{\partial x} E_i, \frac{\partial N^j}{\partial x} E_j \right\rangle} & = {\left\langle \frac{\partial N^i}{\partial y} E_i, \frac{\partial N^j}{\partial y} E_j \right\rangle} > 0 \quad \text{and} \quad {\left\langle \frac{\partial N^i}{\partial x} E_i, \frac{\partial N^j}{\partial y} E_j \right\rangle} = 0
\end{align*}
at the point (where the Einstein summation convention is assumed) iff
\begin{align*}
{\left\langle \del_{\frac{\partial}{\partial x}} N, \del_{\frac{\partial}{\partial x}} N \right\rangle} & = {\left\langle \del_{\frac{\partial}{\partial y}} N, \del_{\frac{\partial}{\partial y}} N \right\rangle} > 0 \quad \text{and} \quad {\left\langle \del_{\frac{\partial}{\partial x}} N, \del_{\frac{\partial}{\partial y}} N \right\rangle} = 0
\end{align*}
at the point iff
\begin{align*}
{\left\langle W \frac{\partial}{\partial x}, W \frac{\partial}{\partial x} \right\rangle} & = {\left\langle W \frac{\partial}{\partial y}, W \frac{\partial}{\partial y} \right\rangle} > 0 \quad \text{and} \quad {\left\langle W \frac{\partial}{\partial x}, W \frac{\partial}{\partial y} \right\rangle} = 0
\end{align*}
at the point iff
\begin{align*}
\III {\left( \frac{\partial}{\partial x}, \frac{\partial}{\partial x} \right)} & = \III {\left( \frac{\partial}{\partial y}, \frac{\partial}{\partial y} \right)} > 0 \quad \text{and} \quad \III {\left( \frac{\partial}{\partial x}, \frac{\partial}{\partial y} \right)} = 0
\end{align*}
at the point iff \( \III {\left( \frac{\partial}{\partial z}, \frac{\partial}{\partial z} \right)} = 0 \) and \( W \ne 0 \) at the point.
\end{proof}

We present the main theorem of this subsection.

\begin{pthrm} \label{T:main}
Let \( S \) be an oriented smooth surface embedded in an oriented flat Riemann-Cartan\/ \( 3 \)-manifold \( M \). Let \( \pi \colon FM \surjto M \) be the oriented orthonormal frame bundle, which has the flat connection \( \omega \in \Omega^1 {\left( FM, \mathfrak{so}(3) \right)} \). Then the following two hold.
\begin{enumerate}
\item (Local version) Let \( s \) be a local \( \omega \)-horizontal smooth section of \( \pi \colon FM \surjto M \) near a point in \( S \). Then the local Gauss map \( n^s \) is conformal at the point iff\/ \( \II \ne 0 \) at the point and either \( \boldsymbol{H} = 0 \) or \( S \) is umbilic at the point.
\item (Global version) If \( S \) is connected, then the local Gauss map with respect to any local \( \omega \)-horizontal smooth section of \( \pi \colon FM \surjto M \) is conformal iff \( S \) is nowhere geodesic and either it is minimal or totally umbilic.
\end{enumerate}
\end{pthrm}
\begin{proof}
This is an immediate corollary of Propositions~\ref{P:TFF-minimal} and~\ref{P:TFF-conformal}.
\end{proof}

\noindent The Weitzenb{\"o}ck version is the following.

\begin{pcor} \label{C:main}
Let \( S \) be an oriented smooth surface embedded in a Weitzenb{\"o}ck\/ \( 3 \)-manifold \( M \) with global smooth frame \( s \). Then the following two hold.
\begin{enumerate}
\item (Local version) The Gauss map \( n^s \colon S \to \mathbb{S}^2 \) is conformal at a point in \( S \) iff\/ \( \II^s \ne 0 \) at the point and either \( \boldsymbol{H}^s = 0 \) or \( S \) is umbilic at the point.
\item (Global version) If \( S \) is connected, then the Gauss map \( n^s \colon S \to \mathbb{S}^2 \) is conformal iff \( S \) is nowhere geodesic and either it is minimal or totally umbilic.
\end{enumerate}
\end{pcor}

This directly implies the following classical result about minimal surfaces. See, e.g., Definition~2.1.8 of~\cite{MR3012474}.

\begin{pcor} \label{C:classical}
Let \( S \) be an oriented smooth surface embedded in\/ \( \mathbb{R}^3 \).
\begin{enumerate}
\item (Local version) The Gauss map \( \vec{n} \colon S \to \mathbb{S}^2 \) is conformal at a point in \( S \) iff\/ \( \bar{\II} \ne 0 \) at the point and either \( \bar{H} = 0 \) or \( S \) is umbilic at the point.
\item (Global version) If \( S \) is connected, then the Gauss map \( \vec{n} \colon S \to \mathbb{S}^2 \) is conformal iff \( S \) is nowhere geodesic and either it is minimal or totally umbilic.
\end{enumerate}
\end{pcor}

\noindent There is a hyperbolic analogue found by Bryant~\cite{MR0955072}: In the hyperbolic \( 3 \)-space \( \mathbb{H}^3 \), if we use the orthogonal geodesic rays toward the \emph{sphere at infinity} to define the \emph{hyperbolic Gauss map} and consider surfaces having constant mean curvature \( 1 \) (or \( 2 \), following the convention of this paper) in place of minimal surfaces, then an analogous statement holds. However, it seems nontrivial whether it would follow from Theorem~\ref{T:main} or Corollary~\ref{C:main}. More generally, it would be interesting if there is a Riemannian analogue.

We end this section by providing several concrete examples, including totally geodesic surfaces, minimal surfaces, totally umbilic surfaces, and surfaces having holomorphic \( \boldsymbol{H} \) in Riemann-Cartan geometry with torsion.

\begin{pex}
A natural example one may consider would be \( \mathbb{R}^3 \) with a nontrivial global smooth frame. Let \( M = \mathbb{R}^3 \) be the Weitzenb{\"o}ck \( 3 \)-manifold with global smooth frame \( {\left( E_1, E_2, E_3 \right)} = {\left( \partial_1, \partial_2, \partial_3 \right)} \cdot g \), where \( {\left( \partial_1, \partial_2, \partial_3 \right)} \) is the standard global smooth frame and \( g = \boldsymbol{e}^{\theta \hat{e}} \colon \mathbb{R}^3 \to \mathrm{SO}(3) \) is the smooth map for some smooth function \( \theta \colon \mathbb{R}^3 \to \mathbb{R} \) and unit vector \( e \in \mathbb{S}^2 \). For convenience, let \( \bar{g} = g^{-1} \). Unless \( \theta \) is constant, \( M \) has torsion: For all cyclic triple \( {\left( i, j, k \right)} \), we have
\begin{align}
\begin{aligned}
T {\left( \partial_i, \partial_j \right)} & = \del_{\partial_i} \partial_j - \del_{\partial_j} \partial_i = {\left( \partial_i \bar{g}^{\ell}_j - \partial_j \bar{g}^{\ell}_i \right)} E_{\ell} = {\left( {\left( - \bar{g} \, \theta_i \hat{e} \right)}{}^{\ell}_j - {\left( - \bar{g} \, \theta_j \hat{e} \right)}{}^{\ell}_i \right)} E_{\ell} \\
& = {\left( - \theta_i \hat{e}^{\ell}_j + \theta_j \hat{e}^{\ell}_i \right)} \partial_{\ell} = - \theta_i \hat{e}^i_j \partial_i + \theta_j \hat{e}^j_i \partial_j + {\left( - \theta_i \hat{e}^k_j + \theta_j \hat{e}^k_i \right)} \partial_k \\
& = \theta_i e^k \partial_i + \theta_j e^k \partial_j - {\left( \theta_i e^i + \theta_j e^j \right)} \partial_k,
\end{aligned}
\end{align}
where \( \theta_i = \partial_i \theta \) and the Einstein summation convention is assumed.

Let \( S = {\left\{ z = 0 \right\}} \subseteq M \) be the \( xy \)-plane with \( N = \partial_3 \). Then the Weingarten map is given by \( W = - \mathrm{d} \bar{g}^i_3 \otimes E_i = {\left( \bar{g} \, \mathrm{d} \theta \hat{e} \right)}{}^i_3 \otimes E_i = \mathrm{d} \theta \hat{e}^i_3 \otimes \partial_i = \mathrm{d} \theta \otimes {\left( e^2 \partial_1 - e^1 \partial_2 \right)} \) by Proposition~\ref{P:Weingarten-map}, where the Einstein summation convention is assumed. That is,
\begin{align}
W & = \begin{pmatrix*}[r]
\theta_x e^2 & \theta_y e^2 \\
- \theta_x e^1 & - \theta_y e^1
\end{pmatrix*} \quad \text{relative to} \ {\left( \partial_1, \partial_2 \right)}. \label{E:W}
\end{align}
We have the following observations.
\begin{enumerate}
\item \( S \) is totally geodesic iff either \( \theta_x = \theta_y = 0 \) on \( S \) or \( e = {\left( 0, 0, \pm 1 \right)} \).
\item \( K = K^{\mathrm{e}} = 0 \) everywhere for any \( \theta \colon \mathbb{R}^3 \to \mathbb{R} \) and \( e \in \mathbb{S}^2 \).
\item \( \boldsymbol{H} = H + \boldsymbol{i} {\star} \tau = {\left( - \theta_y e^1 + \theta_x e^2 \right)} + \boldsymbol{i} {\left( - \theta_x e^1 - \theta_y e^2 \right)} \) by~\eqref{E:W}. This can be also easily found by Proposition~\ref{P:gauge-transformation}, since \( {\left( \partial_1, \partial_2, \partial_3 \right)}{}^{\top} = {\left( \partial_1, \partial_2, 0 \right)} \) and \( {\left( \partial_1, \partial_2, \partial_3 \right)}{}^{\times} = {\left( \partial_2, - \partial_1, 0 \right)} \). We also see that \( S \) is minimal iff it is totally geodesic in this example.
\item Note that
\begin{align}
\begin{aligned}
H_x & = - \theta_{xy} e^1 + \theta_{xx} e^2, & \quad {\left( {\star} \tau \right)}_y & = - \theta_{xy} e^1 - \theta_{yy} e^2, \\
H_y & = - \theta_{yy} e^1 + \theta_{xy} e^2, & \quad - {\left( {\star} \tau \right)}_x & = \theta_{xx} e^1 + \theta_{xy} e^2.
\end{aligned}
\end{align}
Therefore, \( \boldsymbol{H} \) is holomorphic on \( S \) iff either \( e = {\left( 0, 0, \pm 1 \right)} \) or \( \theta |_S \) is harmonic. Moreover, for any holomorphic function \( f \colon S \to \mathbb{C} \), there exists a global smooth frame so that \( \boldsymbol{H} = f \), since \( \boldsymbol{H} = \theta_y + \boldsymbol{i} \theta_x \) if \( e = {\left( -1, 0, 0 \right)} \).
\item By Proposition~\ref{P:SSFF}, we have
\begin{align}
\begin{aligned}
L {\left( \partial_1, \partial_2 \right)} & = - J_S W J_S T_S {\left( \partial_1, \partial_2 \right)} = - J_S W J_S {\left( \theta_x e^3 \partial_1 + \theta_y e^3 \partial_2 \right)} \\
& = - J_S W {\left( \theta_x e^3 \partial_2 - \theta_y e^3 \partial_1 \right)} = 0
\end{aligned}
\end{align}
for any \( \theta \colon \mathbb{R}^3 \to \mathbb{R} \) and \( e \in \mathbb{S}^2 \). The Hopf differential is given by
\begin{align}
\begin{aligned}
\varphi & = {\left\langle \frac{1}{2} {\left( \theta_x - \boldsymbol{i} \theta_y \right)} {\left( e^2 \partial_1 - e^1 \partial_2 \right)}, \frac{1}{2} {\left( \partial_1 - \boldsymbol{i} \partial_2 \right)} \right\rangle} {\left( \mathrm{d} x + \boldsymbol{i} \, \mathrm{d} y \right)}^2 \\
& = \frac{1}{4} {\left( {\left( \theta_y e^1 + \theta_x e^2 \right)} + \boldsymbol{i} {\left( \theta_x e^1 - \theta_y e^2 \right)} \right)} {\left( \mathrm{d} x + \boldsymbol{i} \, \mathrm{d} y \right)}^2.
\end{aligned}
\end{align}
Indeed, \( \boldsymbol{H} \) is holomorphic on \( S \) iff \( \varphi \) is a holomorphic quadratic differential on \( S \), which agrees with Corollary~\ref{C:Hopf-differential}. Also, \( S \) is umbilic exactly at the points where \( \varphi = 0 \), as in Proposition~\ref{P:umbilic}.
\end{enumerate}
\end{pex}

\begin{pex} \label{Ex:catenoid}
Recall that a catenoid in \( \mathbb{R}^3 \) is given by the parametrization
\begin{align}
X(u,v) & = {\left( \cosh v \cos u, \cosh v \sin u, v \right)} \qquad (u \in \mathbb{R} / 2 \boldsymbol{\pi} \mathbb{Z}, \ v \in \mathbb{R}).
\end{align}
Note that \( {\left( \bar{E}_1, \bar{E}_2, N \right)} = {\left( \partial_1, \partial_2, \partial_3 \right)} \cdot G {\left( u, v \right)} \), where \( {\left( \partial_1, \partial_2, \partial_3 \right)} \) is the standard global smooth frame and
\begin{align}
G {\left( u, v \right)} & = \begin{pmatrix*}[c]
- \sin u & \tanh v \cos u & \sech v \cos u \\
\cos u & \tanh v \sin u & \sech v \sin u \\
0 & \sech v & - \tanh v
\end{pmatrix*} \in \mathrm{SO}(3), \label{E:G}
\end{align}
defines a local oriented orthonormal smooth frame \( {\left( \bar{E}_1, \bar{E}_2 \right)} \) for the catenoid.

Motivated by this, let \( M = \mathbb{R}^3 \) be the Weitzenb{\"o}ck \( 3 \)-manifold with global smooth frame \( {\left( E_1, E_2, E_3 \right)} = {\left( \partial_1, \partial_2, \partial_3 \right)} \cdot g \), where \( g \colon \mathbb{R}^3 \to \mathrm{SO}(3) \) is the smooth map defined by \( g(x,y,z) = G(x,y)^{-1} \). Then \( M \) is not torsion-free:
\begin{align*}
T {\left( \partial_1, \partial_2 \right)} & = \del_{\partial_1} \partial_2 - \del_{\partial_2} \partial_1 = {\left( \partial_1 G^i_2 - \partial_2 G^i_1 \right)} E_i \\
& = {\left( - \tanh y \sin x \right)} E_1 + {\left( \tanh y \cos x \right)} E_2 = {\left( \tanh y \right)} \partial_1,
\end{align*}
where the Einstein summation convention is assumed.

Now, let \( S = {\left\{ z = 0 \right\}} \subseteq M \) be the \( xy \)-plane with \( N = \partial_3 \). Then the Weingarten map satisfies
\begin{align*}
W \partial_1 & = - {\left( \partial_1 G^i_3 \right)} E_i = {\left( \sech y \sin x \right)} E_1 - {\left( \sech y \cos x \right)} E_2 = - {\left( \sech y \right)} \partial_1, \\
W \partial_2 & = - {\left( \partial_2 G^i_3 \right)} E_i \\
& = {\left( \sech y \tanh y \cos x \right)} E_1 + {\left( \sech y \tanh y \sin x \right)} E_2 + {\left( \sech^2 y \right)} E_3 \\
& = {\left( \sech y \right)} \partial_2,
\end{align*}
where the Einstein summation convention is assumed. That is,
\begin{align}
W & = \begin{pmatrix*}[c]
- \sech y & 0 \\
0 & \sech y
\end{pmatrix*} \quad \text{relative to} \ {\left( \partial_1, \partial_2 \right)}.
\end{align}
Thus \( K = K^{\mathrm{e}} = - \sech^2 y < 0 \) but \( \boldsymbol{H} = H + \boldsymbol{i} {\star} \tau = 0 \). We see that \( S \) is a minimal surface in \( M \), which is nowhere geodesic and nowhere umbilic. Meanwhile, the Gauss map \( n = {\left( G^1_3, G^2_3, G^3_3 \right)} \colon S \to \mathbb{S}^2 \) is given by
\begin{align}
n(x,y) & = {\left( \sech y \cos x, \sech y \sin x, - \tanh y \right)},
\end{align}
which satisfies \( n_x \cdot n_x = n_y \cdot n_y = \sech^2 y > 0 \) and \( n_x \cdot n_y = 0 \) everywhere, so that it is conformal by Lemma~\ref{L:conformal}. This agrees with Corollary~\ref{C:main}.
\end{pex}

\begin{pex}
Example~\ref{Ex:catenoid} reverses the usual perspective: Instead of observing the catenoid through the standard frame, it observes the plane through a ``catenoid'' frame. Similarly, one can observe the cylinder through a ``catenoid'' frame as follows. Let \( S \) be a cylinder in \( \mathbb{R}^3 \) given by the parametrization
\begin{align}
X(u,v) & = {\left( \cos u, \sin u, v \right)} \qquad (u \in \mathbb{R} / 2 \boldsymbol{\pi} \mathbb{Z}, \ v \in \mathbb{R}).
\end{align}
Note that \( {\left( \bar{E}_1, \bar{E}_2, N \right)} = {\left( \partial_1, \partial_2, \partial_3 \right)} \cdot G' {\left( u, v \right)} \), where \( {\left( \partial_1, \partial_2, \partial_3 \right)} \) is the standard global smooth frame and
\begin{align}
G' {\left( u, v \right)} & = \begin{pmatrix*}[c]
- \sin u & 0 & \cos u \\
\cos u & 0 & \sin u \\
0 & 1 & 0
\end{pmatrix*} \in \mathrm{SO}(3),
\end{align}
defines a local oriented orthonormal smooth frame \( {\left( \bar{E}_1, \bar{E}_2 \right)} \) for the cylinder \( S \).

Let \( g \colon \mathbb{R}^3 \setminus {\left\{ \text{\( z \)-axis} \right\}} \to \mathrm{SO}(3) \) be the smooth map defined by
\begin{align}
g(x,y,z) & = G' G^{-1} {\left( u(x,y), z \right)},
\end{align}
where \( G \) is the matrix in~\eqref{E:G} and \( u(x,y) \) is the angle of \( (x,y) \) in polar coordinates. Let \( M = \mathbb{R}^3 \setminus {\left\{ \text{\( z \)-axis} \right\}} \) be the Weitzenb{\"o}ck \( 3 \)-manifold with global smooth frame \( {\left( E_1, E_2, E_3 \right)} = {\left( \partial_1, \partial_2, \partial_3 \right)} \cdot g \). Note that \( {\left( \bar{E}_1, \bar{E}_2, N \right)} = {\left( E_1, E_2, E_3 \right)} \cdot G(u,v) \). The Weingarten map of \( S \) in \( M \) satisfies
\begin{align*}
W \bar{E}_1 & = - \bar{E}_1 {\left( G^i_3 \right)} E_i = {\left( \sech v \sin u \right)} E_1 - {\left( \sech v \cos u \right)} E_2 = - {\left( \sech v \right)} \bar{E}_1, \\
W \bar{E}_2 & = - \bar{E}_2 {\left( G^i_3 \right)} E_i \\
& = {\left( \sech v \tanh v \cos u \right)} E_1 + {\left( \sech v \tanh v \sin u \right)} E_2 + {\left( \sech^2 v \right)} E_3 \\
& = {\left( \sech v \right)} \bar{E}_2,
\end{align*}
where the Einstein summation convention is assumed. That is,
\begin{align}
W & = \begin{pmatrix*}[c]
- \sech v & 0 \\
0 & \sech v
\end{pmatrix*} \quad \text{relative to} \ {\left( \bar{E}_1, \bar{E}_2 \right)}.
\end{align}
Thus \( K = K^{\mathrm{e}} = - \sech^2 v < 0 \) but \( \boldsymbol{H} = H + \boldsymbol{i} {\star} \tau = 0 \). We see that \( S \) is a minimal surface in \( M \), which is nowhere geodesic and nowhere umbilic. It is also worth noting that \( S \) is not minimal in \( \mathbb{R}^3 \).
\end{pex}

\begin{pex}
We revisit Example~\ref{Ex:Cartan-Schouten}. Let \( S = {\left\{ (x,y,z) \, \middle| \, x^2 + y^2 + z^2 = 1 \right\}} \subseteq M \) be the unit sphere, which can be locally parametrized by
\begin{align}
X {\left( \theta, \varphi \right)} & = {\left( \sin \theta \cos \varphi, \sin \theta \sin \varphi, \cos \theta \right)} \qquad (\theta \in {\left( 0, \boldsymbol{\pi} \right)}, \ \varphi \in \mathbb{R} / 2 \boldsymbol{\pi} \mathbb{Z}).
\end{align}
Note that \( {\left( \bar{E}_1, \bar{E}_2, N \right)} = {\left( \partial_1, \partial_2, \partial_3 \right)} \cdot G {\left( \theta, \varphi \right)} \), where
\begin{align}
G {\left( \theta, \varphi \right)} & = \begin{pmatrix*}[c]
\cos \theta \cos \varphi & - \sin \varphi & \sin \theta \cos \varphi \\
\cos \theta \sin \varphi & \cos \varphi & \sin \theta \sin \varphi \\
- \sin \theta & 0 & \cos \theta
\end{pmatrix*} \in \mathrm{SO}(3),
\end{align}
defines a local oriented orthonormal smooth frame \( {\left( \bar{E}_1, \bar{E}_2 \right)} \) for \( S \). Then the Weingarten map satisfies
\begin{align*}
W \bar{E}_1 & = - \del_{\bar{E}_1} N = - \bar{E}_1 {\left( G^i_3 \right)} \partial_i - \lambda \bar{E}_1 \times N = - \bar{E}_1 + \lambda \bar{E}_2, \\
W \bar{E}_2 & = - \del_{\bar{E}_2} N = - \bar{E}_2 {\left( G^i_3 \right)} \partial_i - \lambda \bar{E}_2 \times N = - \bar{E}_2 - \lambda \bar{E}_1,
\end{align*}
where the Einstein summation convention is assumed. That is,
\begin{align}
W & = \begin{pmatrix*}[r]
-1 & - \lambda \\
\lambda & -1
\end{pmatrix*} \quad \text{relative to} \ {\left( \bar{E}_1, \bar{E}_2 \right)}.
\end{align}
\( W \) is of this form relative to any local oriented orthonormal smooth frame, even at the points \( {\left( 0, 0, \pm 1 \right)} \in S \) by continuity. We see that \( S \) is totally umbilic. Meanwhile, since \( M \) has constant sectional curvature \( - \lambda^2 \), we have \( K = K^{\mathrm{e}} - \lambda^2 = 1 \) by Proposition~\ref{P:Gaussian-curvature}. Therefore, we have \( \int_S K \, \mathrm{d} a = \int_S \mathrm{d} a = 4 \boldsymbol{\pi} = 2 \boldsymbol{\pi} \chi(S) \), which agrees with Proposition~\ref{P:Gauss-Bonnet}.
\end{pex}

\subsubsection*{Acknowledgements}

The author is deeply grateful to the advisor Jinsung Park for invaluable discussions throughout the research. The author appreciates the coadvisor Hyungryul Baik for helpful feedback and comments. The author also thanks Hongjun Lee for careful comments regarding Lemma~\ref{L:Hopf-differential} and Theorem~\ref{M:B}\@. Finally, the author is thankful to the anonymous referee for fruitful suggestions.

\end{document}